\documentclass[a4ppaper,14pt]{amsart}

\usepackage{enumerate}
\usepackage[latin1]{inputenc}
\usepackage{amsmath}
\usepackage{amsthm}
\usepackage{latexsym}
\usepackage{amssymb}
\usepackage{amsmath}
\usepackage{comment}
\usepackage[all]{xy}
\usepackage{calc}
\usepackage{multicol}
\usepackage{slashed}

\newtheorem{thm}{Theorem}[section]
\newtheorem{prop}{Proposition}[section]

\newtheorem{lem}{Lemma}[section]

\newtheorem{rem}{Remark}[section]
\theoremstyle{notation}
\newtheorem*{notation}{Notation}

\newcommand{\R}{\mathbb{R}}
\newcommand{\C}{\mathbb{C}}
\numberwithin{equation}{section}
\newcommand{\N}{\mathbb{N}}

\newcommand{\eps}{\epsilon}

\newcommand{\wto}{\rightharpoonup}

\newcommand{\vertiii}[1]{{\left\vert\kern-0.25ex\left\vert\kern-0.25ex\left\vert #1
\right\vert\kern-0.25ex\right\vert\kern-0.25ex\right\vert}}

\usepackage[textwidth=138mm,
textheight=215mm,
left=30mm,
right=30mm,
top=25.4mm,
bottom=25.4mm,
headheight=2.17cm,
headsep=4mm,
footskip=12mm,
heightrounded,]{geometry}

\usepackage[colorlinks=true,urlcolor=blue,
citecolor=black,linkcolor=black,linktocpage,pdfpagelabels,
bookmarksnumbered,bookmarksopen]{hyperref}

\makeatletter
\newcommand{\leqnomode}{\tagsleft@true}
\newcommand{\reqnomode}{\tagsleft@false}
\makeatother

\begin{document}

\reqnomode

\title{Higher topological type semiclassical states for Sobolev critical Dirac equations with degenerate potential}

\author{Shaowei Chen \and Tianxiang Gou ${}^*$}

\address{Shaowei Chen
\newline \indent School of Mathematical Sciences, Huaqiao University,
\newline \indent Quanzhou 362021, People's Republic of China.}
\email{swchen6@163.com}

\address{Tianxiang Gou
\newline \indent School of Mathematics and Statistics, Xi'an Jiaotong University,
\newline \indent Xi'an, Shaanxi 710049, China.}
\email{tianxiang.gou@xjtu.edu.cn}
\thanks{ ${}^*$ Corresponding author. E-mail address: tianxiang.gou@xjtu.edu.cn.}

\begin{abstract} In this paper, we are concerned with semiclassical states to the following Sobolev critical Dirac equation with degenerate potential,
\begin{align*}
-\textnormal{i} \eps \alpha \cdot \nabla u + a  \beta u + V(x)
u=|u|^{q-2} u + |u| u \quad \mbox{in} \,\, \R^3,
\end{align*}
where $u:\mathbb{R}^3\rightarrow \mathbb{C}^4$, $2<q<3$, $\eps>0$
is a small parameter, $a>0$ is a constant, $\alpha=(\alpha_1,
\alpha_2, \alpha_3)$, $\alpha_j$ and $\beta$ are $4 \times 4$
Pauli-Dirac matrices. We construct an infinite sequence of higher topological type
semiclassical states with higher energies concentrating around the
local minimum points of the degenerate potential $V$.
Here the degeneracy of $V$   means that $|V(x)|<a$ for any $x \in \R^3$ and
$|V(x)|$ may approach $a$ as $|x|$ tends to infinity.
 The solutions are obtained from
a minimax characterization of higher dimensional symmetric linking structure, which correspond
to critical points of the underlying energy functional at energy levels
where compactness condition breaks down.
Our approach is variational, which mainly relies on penalization method and
blow-up arguments along with local type Pohozaev identity.


\end{abstract}

\maketitle

\section{Introduction}

In this paper, we consider the existence and concentration of semiclassical states to the following nonlinear Dirac equation with Sobolev critical exponent,
\begin{align} \label{dirac}
-\textnormal{i} \eps \alpha \cdot \nabla u + a  \beta u + V(x) u=f(x, |u|) u \quad \mbox{in} \,\, \R^3,
\end{align}
where $\eps>0$ is a small parameter, $f(x, |u|)u=|u|^{q-2} u + |u| u$ for $x \in \R^3$ and $2<q<3$, $a>0$ is a constant,
$\alpha=(\alpha_1, \alpha_2, \alpha_3)$, $\alpha_j$
and $\beta$ are $4 \times 4$ Pauli-Dirac matrices defined by
\begin{align*}
\alpha_j=\left(
\begin{array}{cc}
0 & \sigma_j \\
\sigma_j & 0
\end{array}
\right),
\quad
\beta=\left(
\begin{array}{cc}
I &  0\\
0& -I
\end{array}
\right)
\quad \mbox{for} \,\, j=1, 2, 3
\end{align*}
and
\begin{align*}
\sigma_1=\left(
\begin{array}{cc}
0 & 1 \\
1 & 0
\end{array}
\right),
\quad
\sigma_2=\left(
\begin{array}{cc}
0 &  -\textnormal{i}\\
\textnormal{i} & 0
\end{array}
\right),
\quad
\sigma_3=\left(
\begin{array}{cc}
1 &  0\\
0& -1
\end{array}
\right).
\end{align*}
The relevant Sobolev space designated to investigate semiclassical
states to \eqref{dirac} is equivalent to $H^{1/2}(\R^3, \C^4)$,
which is continuously embedded into $L^p(\R^3, \C^4)$ for any $2
\leq p \leq 3$. And $p=3$ is the Sobolev critical exponent.

Equation \eqref{dirac} arises when one seeks for standing waves to the following time-dependent nonlinear Dirac equation,
\begin{align} \label{dt}
- \textnormal{i} \hbar \partial_t \psi
= \textnormal{i} c \hbar \alpha \cdot \nabla \psi - mc^2 \beta \psi -M(x) \psi+g(x, |\psi|)\psi \quad \mbox{in} \,\, \R \times \R^3,
\end{align}
where $\psi$ represents the wave function of the state of the electron, $\hbar$ is the Planck constant,
$c$ is the speed of light, $m$ is the mass of the electron, the external fields $M(x)$ and
$g(x, |\psi|) \psi$ represent nonlinear self-coupling. Here a standing wave to \eqref{dt} is a solution of the form
$$
\psi(t, x)=e^{\frac{\textnormal{i} \mu t}{\hbar}} u(x).
$$
In physics, equation \eqref{dt} is widely used to build
relativistic models of extended particles by means of nonlinear
Dirac fields, see for example \cite{BDr, FLR, FFK, Th} and
references therein.

Regarding the study of semiclassical states to nonlinear
Schr\"odinger equations, it is well known that there already exist
a large number of papers, see for example \cite{ABC, AMN, AR, BJ,
BW1, BW2, CLW, CW, DF1, DF2, JT, JR, Oh1, Oh2, Ra, Wang} and
references therein. However, there are relatively few papers
devoted to the study of semiclassical states to nonlinear Dirac
equations. In \cite{Ding1}, the author initially proved the
existence of semiclassical states to \eqref{dirac} with $V(x)=0$
and $f(x, |u|)u=P(x) |u|^{p-2} u$ for $2<p<3$ clustering near the
global maximum points of $P$. Later, this result was extended to
the case for \eqref{dirac} with competing potentials in
\cite{DiLi,DiRu1,DLR}, where the authors assumed that $V$
satisfies the condition
$$
V(x) \not\equiv 0, \quad \min_{x \in \R^3} V(x) < \liminf_{|x| \to \infty} V(x).
$$
In \cite{DX}, by assuming that there exists a bounded domain
$\Lambda \subset \R^3$ such that
\begin{align*}
\min_{x \in \overline{\Lambda}} V(x)< \min_{x \in \partial \Lambda} V(x),
\end{align*}
the authors established the existence of semiclassical states to
\eqref{dirac} concentrating around the local minimum points of the
potential $V$. The solutions obtained in these papers are of
ground state type and the number of such solutions are finite.
Successively, it was derived in \cite{WZ} that there exist an
unbounded sequence of semiclassical states to \eqref{dirac}  with
$f(x, |u|)u=|u|^{p-2} u$ for $2<p<3$ concentrating around the
local minimum points of the potential $V$. The result in \cite{WZ}
was recently generalized to the Sobolev critical case in
\cite{CG}. All the previous papers mainly concern semiclassical
states to \eqref{dirac} for the case when the potential $V$
fulfills the condition $\sup_{x \in \R^3} |V(x)|<a$. This
condition is helpful to set up an infinitely dimensional
topological linking structure for the corresponding energy
functional and to verify the compactness of the problem. Lately, the
authors in \cite{WZ1} extended the result in \cite{WZ} to the case
when the potential $V$ is degenerate, i.e. $V$ possesses a locally
trapping region, $|V(x)|<a$ for $x \in \R^3$ but $|V(x)|$ may
approach to $a$ as $|x|$ tends to infinity. In this case, the
associated energy functional no longer admits a linking structure
and the treatment of compactness issues becomes complex. Inspired
by the works above, it is interesting to question whether the
result in \cite{WZ1} can be extended to the Sobolve critical case.
This is the motivation of the present paper.

More precisely, for the potential $V$, we formulate the following assumptions.

\begin{enumerate}
\item [($V_1$)] $V \in C^1(\R^3, \R)$ and there are constants $0 <\tau<2$ and $\gamma>0$ such that
$$
a-|V(x)| \geq \frac{\gamma}{1+|x|^{\tau}} \quad \mbox{for any} \,\, x \in \R^3.
$$
\item [($V_2$)] There exists a bounded domain $\mathcal{M} \subset \R^3$ with smooth boundary $\partial \mathcal{M}$ such that
$$
\nabla V(x) \cdot {\bf{n}}(x) >0 \quad \mbox{for any} \ \, x \in \partial \mathcal{M},
$$
where ${\bf{n}}(x)$ denotes the unit outward normal vector to $\partial \mathcal{M}$ at $x$.
\end{enumerate}

Note that the assumption $(V_2)$ is fulfilled if $V$ has an isolated local minimum set,
i.e. $V$ has a local trapping potential well. In order to address our main result, we shall fix some notations.
Under the assumption $(V_2)$, we define the set of critical points of $V$ by
\begin{align} \label{defv}
\mathcal{V}:=\{x \in \mathcal{M}: \nabla V(x)=0\}.
\end{align}
Clearly, $\mathcal{V}$ is a nonempty compact subset of $\mathcal{M}$. Without loss of generality,
we shall assume that $0 \in \mathcal{V}$ throughout the paper. For $\Omega \subset \R^3$, $\eps>0$ and $\delta >0$, we define
$$
\Omega_{\eps}:=\left\{x \in \R^3: \eps x \in \Omega\right\},
$$
and
$$
\Omega^{\delta}:=\left\{x \in \R^3: \mbox{dist}(x, \,
\Omega):=\inf_{y \in \Omega} |x-y| < \delta\right\}.
$$

The main result of this paper reads as follows.

\begin{thm}\label{jgh77rtff11}
Assume $5/2<q<3$, $(V_1)$ and $(V_2)$ hold. Then, for any positive integer $N$,
there exists a constant $\epsilon_N>0$ such that, for any $0<\epsilon<\epsilon_N$,
\eqref{dirac} has at least $N$ pairs of solutions $\pm
u_{j, \epsilon}$ for $1 \leq j \leq  N$. Furthermore, for any
$\delta>0$, there exist $c=c(\delta, N)>0$ and $C=C(\delta, N)>0$
such that
$$
|u_{j, \epsilon}(x)|\leq C\exp\left(-\Big(\frac{c
\,\textnormal{dist}(x,\mathcal{V}^\delta)}{\epsilon}
\Big)^{\frac{2-\tau}{2}}\right), \quad 1 \leq j \leq N,
$$
where $\mathcal{V}$ is defined by \eqref{defv}.
\end{thm}

\begin{rem}
The assumption $5/2<q<3$ stems from Lemma \ref{nvcbvjug8f877fs},
 which is only used to exclude the possibility of blowing-up of semiclassical states to \eqref{dirac} in the Sobolev critical case.
\end{rem}

In comparison with the study carried out in the existing literature,
the new feature of ours lies in establishing infinitely many localized semiclassical states with higher energies to \eqref{dirac} in the Sobolev critical case,
where the potential $V$ is degenerate. In fact, the energy levels of the solutions can be arbitrarily large,
which go beyond the threshold guaranteeing the compactness condition valid in the Sobolev critical case.

Let us now outline the methods involved to prove Theorem \ref{jgh77rtff11}. Note first that if the potential $V$ is degenerate,
then the underlying energy functional does not possess a linking structure
and the verification of the boundedness of the Palais-Smale sequence is not straightforward.
Moreover, the solutions correspond   to critical points of the underlying energy functional
at energy levels exceeding   the threshold   ensuring the compactness condition valid in the Sobolev critical case.
As a consequence, we adapt the ideas from \cite{CLW, WZ1, ZLL} to modify the energy functional
by adding a penalized functional term and truncating the nonlinearity.
By doing this, we can prove that the modified energy functional enjoys a linking structure and satisfies the Palais-Smale condition.
It then leads to the existence of an unbounded sequence of semiclassical states to the modified problem.
To complete the proof, it remains to deduce that the solutions we obtained are indeed ones to the original problem.
For this, we shall make use of some ideas from \cite{CG, CLW} to derive uniform $L^{\infty}$ estimate of the solutions.

The   paper is organized as follows. In Section \ref{pre},
we establish a variational framework for our problem, introduce the modified equation and present some preliminary results.
In Section \ref{exist}, applying an abstract theorem, we deduce the existence of an infinite sequence of semiclassical states to the modified problem.
 In Section \ref{estimate}, we prove Proposition \ref{bcbvhfyfyufuadx},
 by which we conclude that the solutions we obtained in Section \ref{exist} are indeed ones to \eqref{Deps}.
 Section \ref{proof} is devoted to the proof of Theorem \ref{jgh77rtff11}.

\begin{notation}
Throughout the paper, we use the notations $\to$ and $\wto$ to denote the strong convergence and the weak convergence of sequences in associated spaces, respectively.
We use letters $c$ and $C$ to denote generic positive constants, whose values may change from line to line.
In addition, $B_R(x)$ represents the open ball in $\R^3$ with center at $x \in \R^3$ and radius $R>0$,
and $o_n(1)$ stands for quantities which tend to zero as $n \to \infty$. For a function $u$, we denote by $\overline{u}$ the conjugate function of $u$.
\end{notation}

\section{Preliminaries} \label{pre}

In this section, let us present some preliminaries used to establish our main result.
Firstly, by making a change of variable $x \to \eps x$,  we may rewrite \eqref{dirac} as
\begin{align} \label{Deps}
-\textnormal{i} \alpha \cdot \nabla u + a  \beta u + V(\eps x)
u=|u|^{q-2} u + |u| u \quad \mbox{in} \,\, \R^3.
\end{align}
In what follows, for any $p \geq
1$, we denote by $\|\cdot\|_{L^p(\R^3)}$ the usual norm in
$L^p(\R^3, \C^4)$. Let us define
$$
H_a:=-\textnormal{i} \alpha \cdot \nabla +a  \beta.
$$
It is obvious that $H_a$ is self-adjoint on $L^2(\R^3, \C^4)$
with domain $\mathcal{D}(H_a)=H^1(\R^3, \C^4)$. Under the
assumption $(V_1)$, we see that $\sigma(H_a)=\sigma_c(H_a)=\R
\backslash (-a, a)$, where $\sigma(H_a)$ and $\sigma_c(H_a)$ stand
for the spectrum and the continuous spectrum of the operator $H_a$
on $L^2(\R^3, \C^4)$, respectively. As a consequence, $L^2(\R^3, \C^4)$ has the
following orthogonal decomposition,
$$
L^2(\R^3, \C^4)=L^+ \oplus L^-,
$$
where $H_a$ is positive definite on $L^+$ and it is negative definite on $L^-$.

Let $|H_a|$ be the absolute value of $H_a$ and $|H_a|^{\frac 12}$
be its square root. We now introduce a space
$E=\mathcal{D}(|H_a|^{\frac 12})$ endowed with the following inner product and norm
$$
\langle u,\, v \rangle :=(|H_a|^{\frac 12} u,\, |H_a|^{\frac 12}
v)_2, \quad  \|u\|:=\langle u,\, u
\rangle ^{\frac 12}\quad \mbox{for any} \,\, u,v \in E,
$$
where $(\cdot, \cdot)_2$ denotes the inner product in
$L^2(\R^3, \C^4)$. From \cite[Lemma 7.4]{Ding}, we have that $E
\cong H^{\frac 12}(\R^3, \C^4)$ and the norm $\|\cdot\|$ on $E$ is
equivalent to the usual norm in $H^{\frac 12}(\R^3, \C^4)$. In
addition, we know that the embedding $E \hookrightarrow L^p(\R^3, \C^4)$ is
continuous for any $2 \leq p \leq 3$ and it is locally compact
for any $ 1 \leq p <3$. Since $\sigma(H_a)=\R \backslash (-a, a)$,
then
\begin{align} \label{l2}
a \|u\|_{L^2(\R^3)}^2 \leq \|u\|^2 \quad \mbox{for any} \, \, u
\in E.
\end{align}
Based upon the orthogonal decomposition in $L^2(\R^3, \C^4)$, we have that $E$ possesses the following decomposition,
$$
E=E^+ \oplus E^-,
$$
where $E^+=E \cap L^+$ and $E^-=E \cap L^-$. Thus, for any $u \in
E$, there holds that $u=u^+ +u^-$ for $u^+ \in E^+$ and $u^- \in
E^-$, and the sum is orthogonal with respect to both $\langle
\cdot, \cdot \rangle$ and $(\cdot, \cdot)_2$. Moreover, it follows
from \cite[Proposition 2.1]{DX} that there exists a constant
$d_q>0$ such that for any $q\in[2,3],$
\begin{align} \label{lp}
d_q\|u^{\pm}\|_{L^q(\R^3)} \leq  \|u\| \quad \mbox{for any} \,\, u
\in E.
\end{align}

Let $\varphi \in C^{\infty}(\R^+, [0, 1])$ be such that
$\varphi(t)=1$ if $ 0 \leq t \leq 1$, $\varphi(t)=0$ if $t \geq 2$
and $\varphi'(t) \leq 0$ for any $t \geq 0$. Setting
$b_{\eps}(t)=\varphi(\eps t)$ and $m_{\eps}(t)=\int_{0}^{t}
b_{\eps}(s) \, ds$ for any $t \geq 0$, we then have that following statement.

\begin{lem} \cite[Proposition 2.1] {CLW} \label{mprop}
The functions $b_{\eps}$ and $m_{\eps}$ satisfy the following
properties.
\begin{enumerate}
\item[$(\textnormal{i})$] $t b_{\eps}(t) \leq m_{\eps}(t) \leq t$
for any $t \geq 0$. \item[$(\textnormal{ii})$] There exists $c>0$
such that $m_{\eps}(t) \leq c  / \eps $ for any $t \geq 0$. If $ 0
\leq t  \leq 1 / \eps$, then $b_{\eps}(t)=1$ and $m_{\eps}(t)=t$.
\end{enumerate}
\end{lem}

According to $(V_2)$, we have that there is $\delta_0 >0$ such
that, for any $y \in \mathcal{M}^{\delta_0}$,  if $B(y, \delta_0)
\backslash\mathcal{M}  \neq \emptyset$, there holds that
\begin{align*} 
\inf_{x \in B(y, \, \delta_0) \backslash \mathcal{M}} \nabla
V(x)\cdot \nabla \mbox{dist}(x, \mathcal{M})>0.
\end{align*}
From $(V_1)$, we know that there exist $0<\theta<1$ and $R_0>0$ large enough determined later
such that $\mathcal{M}^{\delta_0+1} \subset B_{R_0/2}(0)$ and $ |V(x)| \leq \theta a $ for any $|x| \leq R_0 +1.$

Let $\zeta_1 \in C^{\infty}(\R^+, [0, 1])$ be a cut-off function
such that $\zeta(t)=0$ if $t \leq 0$, $\zeta(t) >0$ if $t >0$ and
$\zeta(t) =1$ if $t \geq 1$. We then set
$$
\chi_1(x):=\zeta_1(|x|-R_0) \quad \mbox{for any} \,\, x \in \R^3.
$$
For a function $\phi : \R^3 \to \R$ with
$$
\phi(x)=\frac{1}{1+|x|^4} \quad \mbox{for any} \,\, x \in \R^3,
$$
we define functions $\xi, \xi_1: \R^3 \times \R \to \R$ by
\begin{align*}
\xi(x, t):=\left\{
\begin{aligned}
&0,  \,\, \, \hspace{4cm}  t \leq  \phi(x),\\
&\frac{1}{\phi(x)}\left(t-\phi(x)\right)^2, \quad  \ \ \, \phi(x) <t <2 \phi(x), \\
&2t-3 \phi(x), \,\, \, \, \hspace{2.5cm} t \geq 2\phi(x),
\end{aligned}
\right.
\end{align*}
and
\begin{align*}
\xi_1(x, t):= \int_{-\infty}^t \xi(x, s) \, ds =\left\{
\begin{aligned}
&0,  \, \, \quad  \  \hspace{3.5cm}  t \leq  \phi(x),\\
&\frac{1}{3\phi(x)}\left(t-\phi(x)\right)^3, \quad \, \phi(x) <t <2 \phi(x), \\
&t^2-3 \phi(x) t + \frac 7 3 \phi(x)^2, \quad \qquad  t \geq
2\phi(x).
\end{aligned}
\right.
\end{align*}
We now define the associated penalized functional $Q_{\eps}: E \to \R$ by
\begin{align} \label{defq}
Q_{\eps}(u):= \frac 18 \int_{\R^3} \chi_1(\eps x) V(\eps x) \xi_1(x, |u|) \, dx.
\end{align}
It is standard to check that $Q_{\eps} \in C^1(E, \R)$ and
\begin{align*} 
Q_{\eps}'(u) v=\frac 18 \textnormal{Re} \int_{\R^3} \chi_1(\eps x) V(\eps x) \tilde{\xi}(x, |u|) u \cdot \overline{\psi} \, dx \quad \mbox{for any} \,\, v \in E,
\end{align*}
where $u\cdot v=\sum^4_{i=1}u_iv_i$ for
$u=(u_1,\cdots,u_4)\in\C^4$ and $v=(v_1,\cdots,v_4)\in\C^4$ and
\begin{align*}
\tilde{\xi}(x, t):=\frac{\xi(x, t)}{t}=\left\{
\begin{aligned}
&0,  \quad  \hspace{4cm}  t \leq  \phi(x),\\
&\frac{1}{t\phi(x)}\left(t-\phi(x)\right)^2, \qquad  \phi(x) <t <2 \phi(x), \\
&2- \frac{3 \phi(x)}{t}, \quad  \hspace{2.5cm} \,\,\, t \geq 2\phi(x).
\end{aligned}
\right.
\end{align*}

Let $\zeta_2\in C^{\infty}(\R^+, [0, 1])$ be a cut-off function
such that $\zeta_2(t)=0$ if $t \leq 0$, $\zeta_2(t) >0$ if $t >0$ and
$\zeta_2(t) =1$ if $t \geq \delta_0$, and $\zeta'_2(t) \geq 0$ for any
$t \geq 0$. We then set $\chi_2(x)=\zeta_2(\mbox{dist}(x, \, \mathcal{M}))$
and
\begin{align*}
g_{\eps}(x, t)=\min\left\{h_{\eps}(t), \, \phi(x)\right\}  \quad
\mbox{for any} \,\, x \in \R^3, t \geq 0,
\end{align*}
where
\begin{align*}
h_{\eps}(t)=t^{q-2} + \frac{q}{3} t^{q-2} \left(m_{\eps}(t^2)
\right)^{\frac{3-q}{2}} + \frac{3-q}{3} t^{q} \left(m_{\eps}(t^2)
\right)^{\frac{3-q}{2} -1} b_{\eps}(t^2).
\end{align*}
Observe that
$$
H_{\eps}(t)=\int_{0}^t h_{\eps}(s)s \, ds =\frac{t^q}{q} +
\frac{t^q}{3}\left(m_{\eps}(t^2)\right)^{\frac{3-q}{2}},
$$
then it holds that
\begin{align} \label{arh}
h_{\eps}(t)t^2-q H_{\eps}(t) \geq 0 \quad \mbox {for any} \,\, t
\geq 0.
\end{align}
Thus, from the definition of $g_{\eps}$, we know that
\begin{align}  \label{arg}
g_{\eps}(x, t)t^2-2 G_{\eps}(x, t) \geq 0 \quad  \mbox{for any} \,\, x \in \R^3, t
\geq 0,
\end{align}
where $G_{\eps}(x, t)=\int_{0}^t g_{\eps}(x, s)s \, ds$. Let us define
\begin{align*} 
f_\epsilon(x, t)=\left(1- \chi_2(\epsilon x)\right)
h_{\eps}(t)+\chi_2(\epsilon x){g_{\eps}}(x, t) \quad \mbox{for any}
\,\, x \in \R^3, t \geq 0,
\end{align*}
then
\begin{align*}
\begin{split}
F_\epsilon(x, t)&=\int_{0}^t f_\epsilon(x, s) s \, ds \\ 
& =\left(1- \chi_2(\epsilon x)\right) \left( \frac {t^q}{q} +
\frac{t^{q}}{3} \left(m_{\eps}(t^2) \right)^{\frac{3-q}{2}}
\right) + \chi_2(\epsilon x)  G_{\eps}(x, t).
\end{split}
\end{align*}
We now introduce a modified energy functional $\Gamma_{\eps}:E \to \R$ as
\begin{align} \label{functional}
\Gamma_{\eps}(u)=\frac 12 \left(\|u^+\|^2-\|u^-\|^2 \right) +
\frac 12 \int_{\R^3} V_{\eps}(x) |u|^2 \, dx - Q_{\eps}(u)-\int_{\R^3}
F_{\eps}(x, |u|) \, dx.
\end{align}
It is not hard to deduce that $\Gamma_{\eps}$ is of class $C^1$ on
$E$ and
$$
\Gamma_{\eps}'(u) v= \mbox{Re} \int_{\R^3} \Big(H_a u +
V_{\eps}(x) u -\frac 18 \chi_1(\eps x) V(\eps x) \tilde{\xi}(x, |u|) u-f_{\eps}(x, |u|) u \Big) \cdot \overline{v} \, dx
\,\, \mbox{for any} \,\, v \in E.
$$
Therefore, critical points of $\Gamma_{\eps}$ are solutions of the equation
\begin{align} \label{mdirac}
-\textnormal{i} \alpha \cdot \nabla u + a  \beta u + V(\eps x)
u-\frac 18 \chi_1(\eps x) V(\eps x) \tilde{\xi}(x, |u|) u=f_{\eps}(x, |u|) u.
\end{align}

\section{Existence of semiclassical states} \label{exist}

In this section, our principal aim is to establish the existence
of an infinite sequence of semiclassical states to \eqref{mdirac} with higher energies. To do so, let us first introduce
the following abstract theorem, which is adapted to guarantee the existence of the solutions to \eqref{mdirac}.

\begin{thm} \label{theorem} \cite[Theorem 4.6]{BD}
Let $(X, \|\cdot\|_X)$ be a separable reflexive Banach space with
orthogonal decomposition $X=X^+ \oplus X^-$,  $I \in C^1(X, \R)$
be even and satisfy the Palais-Smale condition. Assume the
following assumptions hold.
\begin{enumerate}
\item [$(\textnormal{i})$] There exist a constant $r_0>0$ and a constant $\rho>0$ such that
$$
\inf_{u \in X^+, \, \|u\|_X=r_0} I(u) \geq \rho.
$$
\item[$(\textnormal{ii})$] There exist a $N$-dimensional subspace
$X_N=\textnormal{span}\{e_1, \cdots, e_N\} \subset X^+$ and a
constant $R_N> r_0$ such that
$$
\sup_{u \in X^- \oplus X_N} I(u)< \infty, \quad \sup_{u \in X^-
\oplus X_N, \, \|u\|_X \geq R_N} I(u)\leq 0.
$$
\item[$(\textnormal{iii})$] Let $\mathcal{S}$ be a countable dense subset of
the dual space of $X^-$, $\mathcal{P}$ be a family of semi-norms
on $X$ consisting of all semi-norms
$$
p_s: X \to \R, \quad p_s(u)=|s(u^-)| + \|u^+\|_X, \quad u=u^+ +
u^- \in X, \,\, s \in \mathcal{S}.
$$
If we denote the induced topology by $(X,
\mathcal{T}_{\mathcal{P}})$ and denote the weak${}^*$ topology on
$X^*$ by $(X^*, \mathcal{T}_{w^*})$, there hold that, for any
$c \in \R$, $I_c=\{u \in X: I(u) \geq c\}$ is
$\mathcal{T}_{\mathcal{P}}$-closed and $I':(I_c,
\mathcal{T}_{\mathcal{P}}) \to (X^*, \mathcal{T}_{w^*})$ is
continuous.
\end{enumerate}
Then $I$ has as least $N$ pairs of critical points $\pm u_j$ with
$$
I(u_j) \in \left[\rho, \sup_{X^- \oplus X_N} I(u)\right] \quad
\mbox{for any} \,\, 1 \leq j \leq N.
$$
\end{thm}

In the following, we shall show that the energy functional $\Gamma_{\eps}$ fulfills
the conditions of Theorem \ref{theorem}. Let us start with verifying the following compactness condition.

\begin{lem} \label{ps}
There exists a constant $\eps_0>0$ such that, for any $0<\eps<\eps_0$, the energy functional $\Gamma_{\eps}$ satisfies the Palais-Smale condition on $E$.
\end{lem}
\begin{proof}
Supposing $\{u_n\} \subset E$ is a Palais-Smale sequence for
$\Gamma_{\eps}$, i.e. $\{\Gamma_{\eps}(u_n)\} \subset \R$ is
bounded and $\Gamma_{\eps}'(u_n)=o_n(1)$, we shall prove that
$\{u_n\}$ admits a convergent subsequence in $E$. To do this, we first need to
derive that $\{u_n\}$ is bounded in $E$. Note that $u_n=u_n^+ + u_n^-$ and
\begin{align} \label{boundedness}
\begin{split}
o_n(1)\|u_n\|  & \geq  \Gamma_{\eps}'(u_n)(u_n^+-u_n^-) \\
&=\|u_n\|^2 + \mbox{Re} \int_{\R^3} V(\eps x) u_n \cdot \overline{u_n^+-u_n^-} \, dx - \mbox{Re} \int_{\R^3} f_{\eps}(x, |u_n|) u_n \cdot \overline{u_n^+-u_n^-} \, dx \\
& \quad - \frac 18 \mbox{Re} \int_{\R^3} \chi_1(\eps x) V(\eps x) \tilde{\xi}(x, |u_n|) u_n \cdot \overline{u_n^+-u_n^-} \, dx \\
& = \|u_n\|^2 -\mbox{Re} \int_{\R^3}\left(\left(1 -\chi_2(\eps x)\right) h_{\eps}(|u_n|)+\chi_2(\eps x) g_{\eps}(x, |u_n|) \right)  u_n \cdot \overline{u_n^+-u_n^-} \, dx  \\
& \quad +\int_{\R^3} V(\eps x) \left( 1-\frac 18\chi_1(\eps x)
\tilde{\xi}(x, |u_n|) \right)
u_n \cdot \overline{u_n^+-u_n^-}\,
dx.
\end{split}
\end{align}
We are now going to estimate the last two terms in the right hand side of \eqref{boundedness}.
Let us begin with treating the second term in the right hand side of \eqref{boundedness}.
Since $0 \in \mathcal{M}$ and $g_{\eps}(x, t) \leq \phi(x)$ for any $x \in \R^3$ and $t \in \R$,
see the definition of $g_{\eps}$, it then follows from the definition of $\chi_2$, H\"older's inequality and \eqref{l2} that
\begin{align} \label{gest}
\left| \int_{\R^3} \chi_2(\eps x) g_{\eps}(x, |u_n|) u_n \cdot \overline{u_n^+-u_n^-} \, dx \right| \leq c \eps^4 \|u_n\|^2.
\end{align}
From the definition of $h_{\eps}$ and Lemma \ref{mprop}, we know that $h_{\eps}(t) \to 0$ as $t \to 0^+$.
This means that, for any $\eps>0$, there is a constant $r>0$ such that $h_{\eps}(t) \leq \eps$  for any $0 \leq t \leq r$. Let us define
$$
\Omega_{n, 1}:=\left\{x \in \R^3: |u_n(x)| > r\right\}.
$$
In view of \eqref{l2}, \eqref{lp} and H\"older's inequality, then
\begin{align} \label{esth}
\begin{split}
&\left| \int_{\R^3} \left(1 -\chi_2(\eps x)\right) h_{\eps}(|u_n|) u_n \cdot \overline{u_n^+-u_n^-} \, dx \right| \\
&\leq \frac{\eps}{a} \|u_n\|^2 +  \int_{\Omega_{n, 1}}
\left(1 -\chi_2(\eps x)\right) h_{\eps}(|u_n|) |u_n| |u_n^+-u_n^-| \, dx \\
& \leq \frac{\eps}{a} \|u_n\|^2  +\frac{1}{d_3}
\left(\int_{\Omega_{n, 1}}\left(1 -\chi_2(\eps x)\right)
\left(h_{\eps}(|u_n|) |u_n| \right)^{\frac 32}\, dx \right)^{\frac
23} \|u_n^+-u_n^-\|.
\end{split}
\end{align}
Using again the definition of $h_{\eps}$ and Lemma \ref{mprop}, we get that
there is a constant $c_r>0$ such that $h_{\eps}(t) \leq c_r t$ for any $ t
>r$. This then shows that
\begin{align*}
(h_{\eps}(t)t)^{\frac 3 2} \leq c_{r}^{\frac 12} h_{\eps}(t) t^2
\quad \mbox{for any} \,\, t >r.
\end{align*}
On the other hand, from \eqref{arh}, we know that
\begin{align*}
\frac 12 h_{\eps}(t) t^2- H_{\eps}(t) \geq \frac{q -2}{2q}
h_{\eps}(t)t^2 \quad \mbox{for any} \,\, t\geq 0.
\end{align*}
Thus there holds that
\begin{align} \label{hHar}
(h_{\eps}(t)t)^{\frac 3 2} \leq \frac{2q c_r^{\frac 12}}{q-2}\left(\frac 12 h_{\eps}(t) t^2-
H_{\eps}(t)\right) \quad \mbox{for any} \,\, t >r.
\end{align}
Notice that $|V(x)| \leq a$ for any $x \in \R^3$. By the
definitions of $\xi_1$, $\tilde{\xi}$, H\"older's inequality and
\eqref{l2}, it is easy to see that there exists a constant $c>0$
such that
\begin{align} \label{qest}
\begin{split}
\hspace{-0.5cm}\frac 12 Q_{\eps}'(u_n) u_n-Q_{\eps}(u_n)
&= \frac 1 8 \int_{\R^3} \chi_1(\eps x) V(\eps x) \left(\frac 12 \tilde{\xi}(x, |u_n|) |u_n|^2 - \xi_1(x, |u_n|)\right) \, dx  \\
& \geq -c\int_{\R^3} |u_n| \phi \, dx  \geq  -c \|u_n\|.
\end{split}
\end{align}
Observe that
\begin{align*}
c(1 + \|u_n\|) \geq \Gamma_{\eps}(u_n) -\frac 12 \Gamma_{\eps}'(u_n)u_n &=\int_{\R^3} \left(\frac 12 f_{\eps}(x, |u_n|) |u_n|^2 -F_{\eps}(x, |u_n|)\right) \, dx \\
& \quad + \frac 12 Q_{\eps}'(u_n) u_n-Q_{\eps}(u_n).
\end{align*}
It then follows from \eqref{arg}, \eqref{hHar} and \eqref{qest} that
\begin{align*}
\int_{\Omega_{n, 1}} \left(1-\chi_2(\eps x)\right)\left(h_{\eps}(|u_n|) |u_n|\right)^{\frac 32} \, dx & \leq  c \int_{\Omega_{n, 1}}  \left(1 -\chi_2(\eps x)\right)\left(\frac 12 h_{\eps}(|u_n|) |u_n|^2 -H_{\eps}(|u_n|)\right) \, dx \\
& \leq  c\int_{\R^3} \left(\frac 12 f_{\eps}(x, |u_n|) |u_n|^2 -F_{\eps}(x, |u_n|)\right) \, dx\\
& \leq c(1 + \|u_n\|)  +Q_{\eps}(u_n) - \frac 12 Q_{\eps}'(u_n) u_n\\
& \leq c(1 + \|u_n\|).
\end{align*}
Using \eqref{esth}, we now have that
\begin{align*}
\hspace{-0.5cm}\left| \int_{\R^3} \left(1 -\chi_2(\eps x)\right) h_{\eps}(|u_n|) u_n \cdot \overline{u_n^+-u_n^-} \, dx \right|  \leq \frac{\eps}{a} \|u_n\|^2  + c(1 + \|u_n\|)^{\frac 2 3} \|u_n\|.
\end{align*}
This along with \eqref{gest} gives rise to
\begin{align} \label{esth11}
\begin{split}
&-\int_{\R^3}\left(\left(1 -\chi_2(\eps x)\right) h_{\eps}(|u_n|)+\chi_2(\eps x) g_{\eps}(x, |u_n|) \right)  u_n \cdot \overline{u_n^+-u_n^-} \, dx \\
& \geq -\left( c \eps^4+\frac {\eps}{a}\right) \|u_n\|^2 - c(1 + \|u_n\|)^{\frac 2 3} \|u_n\|.
\end{split}
\end{align}
We next deal with the last term in the right hand side of \eqref{boundedness}.
To do this, we define
$$
\Omega_{n, 2}:=\left\{x \in \R^3: |u_n(x)| \geq 3 \phi(x) \right\}.
$$
Observe that
\begin{align*}
&\int_{\R^3}| V(\eps x)| \left( 1-\frac 18\chi_1(\eps x)
\tilde{\xi}(x, |u_n|) \right) |u_n| \cdot |u_n^+-u_n^-|\,
dx \\
&\leq \int_{\R^3}| V(\eps x)| \left( 1-\frac 18 \tilde{\xi}(x,
|u_n|) \right)\chi_1(\eps x) |u_n| \cdot |u_n^+-u_n^-|\,
dx+\int_{\R^3}| V(\eps x)| \left( 1-\chi_1(\eps x) \right) |u_n|
\cdot |u_n^+-u_n^-|\, dx.
\end{align*}
We are now going to estimate each term in the right hand side of the inequality above. In light of the definition of $\tilde{\xi}$, then $ 1 \leq
\tilde{\xi}(x, |u_n(x)|) \leq 2$ for any $x \in \Omega_{n, 2}$, Therefore, we get that
\begin{align*} 
&\int_{\R^3} |V(\eps x)| \left( 1-\frac 18 \tilde{\xi}(x,
|u_n|) \right)\chi_1(\eps x) |u_n| \cdot |u_n^+-u_n^-|\,
dx \\
&=\int_{\Omega_{n, 2}} | V(\eps x)| \left( 1-\frac 18
\tilde{\xi}(x, |u_n|) \right)\chi_1(\eps x) |u_n| \cdot
|u_n^+-u_n^-|\, dx\\
&\quad + \int_{\R^3 \backslash\Omega_{n, 2}} | V(\eps x)| \left(
1-\frac 18 \tilde{\xi}(x, |u_n|) \right)\chi_1(\eps x) |u_n| \cdot
|u_n^+-u_n^-|\, dx\\
& \leq \frac{7a}{8}\int_{\Omega_{n, 2}} |u_n| \cdot |u_n^+-u_n^-|\, dx + 3a\int_{\R^3 \backslash\Omega_{n,
2}}\phi \cdot |u_n^+-u_n^-|\, dx.
\end{align*}
In addition, notice that  $ |V(x)| \leq \theta a $ if $1-\chi_1(  x)\neq0$, then
\begin{align*} 
&\int_{\R^3}| V(\eps x)| \left( 1-\chi_1(\eps x) \right) |u_n|
\cdot |u_n^+-u_n^-|\, dx \leq  \theta a \int_{\Omega_{n, 2}} |u_n|
\cdot |u_n^+-u_n^-|\, dx + 3a\int_{\R^3\setminus\Omega_{n, 2}} \phi \cdot |u_n^+-u_n^-|\,
dx.
\end{align*}
As a consequence, we derive that
\begin{align}\label{jjjnhyytggff}
\begin{split}
\int_{\R^3}| V(\eps x)| \left( 1-\frac 18\chi_1(\eps x)
\tilde{\xi}(x, |u_n|) \right) |u_n| \cdot |u_n^+-u_n^-|\,
dx 
\leq \min\left\{\frac {7} {8}, \theta\right\}\|u_n\|^2+c\|u_n\|,
\end{split}
\end{align}
where we used H\"older's inequality and \eqref{l2}. This together with \eqref{boundedness} and
\eqref{esth11} yields that
$$
\min \left\{1- \theta, \frac 1 8\right\} \|u\|^2 -c \left(\frac{\eps} {a}
+ \eps^4\right) \|u_n\|^2 \leq c \left(1 + \|u_n\| \right)^{\frac 2 3} \|u_n\|+ c( \|u_n\|+1).
$$
Thanks to $0<\theta <1$,  then there exists a constant $\eps_0>0$ such that, for any $0<\eps< \eps_0$, $\{u_n\} $
is bounded in $E$. As a consequence, we know that there exists $u \in E$ such that $u_n \wto u$ in $E$ as
$n \to \infty$. In addition, we find that $\Gamma_{\eps}'(u)=0$.

We now prove that, up to a subsequence, $u_n \to u$ in $E$ as $n \to \infty$. Letting
$z_n=u_n-u$, we then have that
$$
\Gamma_{\eps}'(u_n)(z_n^+-z_n^-)=o_n(1), \quad
\Gamma_{\eps}'(u)(z_n^+-z_n^-)=0.
$$
Thus there holds that
\begin{align*} 
o_n(1)&=\Gamma_{\eps}'(u_n)(z_n^+-z_n^-)-\Gamma_{\eps}'(u)\left(z_n^+-z_n^-\right) \\
&=\|z_n\|^2 + \mbox{Re} \int_{\R^3} V(\eps x) z_n \cdot \overline{\left(z_n^+-z_n^-\right)} \, dx\\
& \quad -\mbox{Re} \int_{\R^3} \left(1 -\chi_2(\eps x)\right) \left(h_{\eps}(|u_n|) u_n-h_{\eps}(|u|)u \right) \cdot \overline{\left(z_n^+-z_n^-\right)} \,dx \\
& \quad -\mbox{Re} \int_{\R^3} \chi_2(\eps x) \left(g_{\eps}(x, |u_n|) u_n-g_{\eps}(x, |u|)u \right) \cdot \overline{\left(z_n^+-z_n^-\right)} \,dx \\
& \quad -\frac 18\mbox{Re} \int_{\R^3} \chi_1(\eps x) V(\eps x) \left(\tilde{\xi}(x, |u_n|) u_n -  \tilde{\xi}(x, |u|) u \right)\cdot \overline{\left(z_n^+-z_n^-\right)} \, dx,
\end{align*}
from which we further get that
\begin{align} \label{conv}
\begin{split}
o_n(1)&=\|z_n\|^2 + \mbox{Re} \int_{\R^3} V(\eps x) z_n \cdot \overline{\left(z_n^+-z_n^-\right)} \, dx\\
& \quad -\mbox{Re} \int_{\R^3} \left(1 -\chi_2(\eps x)\right) \left(h_{\eps}(|u_n|) u_n-h_{\eps}(|u|)u \right) \cdot \overline{\left(z_n^+-z_n^-\right)} \,dx \\
& \quad -\mbox{Re} \int_{\R^3} \chi_2(\eps x) g_{\eps}(x, |u_n|) z_n \cdot \overline{\left(z_n^+-z_n^-\right)} \, dx\\
& \quad - \mbox{Re} \int_{\R^3} \chi_2(\eps x) \left(g_{\eps}(x, |u_n|)-g_{\eps}(x, |u|)\right)u \cdot \overline{\left(z_n^+-z_n^-\right)}
\, dx \\
& \quad -\frac 18\mbox{Re} \int_{\R^3} \chi_1(\eps x) V(\eps x) \tilde{\xi}(x, |u_n|) z_n\cdot \overline{\left(z_n^+-z_n^-\right)} \, dx \\
& \quad -\frac 18\mbox{Re} \int_{\R^3} \chi_1(\eps x) V(\eps x)
\big(\tilde{\xi}(x, |u_n|) - \tilde{\xi}(x, |u|)\big) u\cdot
\overline{\left(z_n^+-z_n^-\right)} \, dx.
\end{split}
\end{align}
Due to $u_n \wto  u$ in $E$ as $n \to \infty$, by the definitions of $g_{\eps}$ and $\tilde{\xi}$, we then
know that
$$
\left(g_{\eps}(x, |u_n|)-g_{\eps}(x, |u|)\right)\left(z_n^+-z_n^-\right) \wto 0 \quad \mbox{in}\,\, L^2(\R^3, \C^4) \,\,\mbox{as}\,\, n \to \infty
$$
and
$$
(\tilde{\xi}(x, |u_n|) - \tilde{\xi}(x, |u|))\left(z_n^+-z_n^-\right) \wto 0 \quad \mbox{in}\,\, L^2(\R^3, \C^4) \,\,\mbox{as}\,\, n \to \infty.
$$
This then leads to
$$
\left|\int_{\R^3} \chi_2(\eps x) \left(g_{\eps}(x, |u_n|)
-g_{\eps}(x, |u|)\right)u \cdot \overline{\left(z_n^+-z_n^-\right)} \,
dx \right|=o_n(1)
$$
and
$$
\left|\int_{\R^3} \chi_1(\eps x) V(\eps x) \big(\tilde{\xi}(x, |u_n|) - \tilde{\xi}(x, |u|)\big) u\cdot \overline{\left(z_n^+-z_n^-\right)} \, dx \right|=o_n(1).
$$
Moreover, by the definition of $g_{\eps}$, it holds that
$$
\left|\int_{\R^3} \chi_2(\eps x) g_{\eps}(x, |u_n|) z_n \cdot \overline{\left(z_n^+-z_n^-\right)} \, dx \right| =o_n(1).
$$
Hence, from \eqref{conv}, we arrive at
\begin{align} \label{sconv}
\begin{split}
&\|z_n\|^2 + \int_{\R^3} V(\eps x) \left(1 - \frac 18 \chi_1(\eps x) \tilde{\xi}(x ,|u_n|)\right) z_n \cdot \overline{\left(z_n^+-z_n^-\right)} \, dx  \\
&\leq \left| \int_{\R^3} \left(1 -\chi_2(\eps x)\right)
\left(h_{\eps}(|u_n|) u_n-h(|u|)u \right) \cdot
\overline{\left(z_n^+-z_n^-\right)} \,dx \right|  + o_n(1).
\end{split}
\end{align}
We are now ready to estimate the term in the right hand side of \eqref{sconv}. Taking into account the mean value theorem, we get that there is a
function $\theta_n$ such that
\begin{align*}
&\left|\int_{\R^3} \left(1 -\chi_2(\eps x)\right) \left(|u_n|^{q-2} u_n-|u|^{q-2}u \right) \cdot \overline{\left(z_n^+-z_n^-\right)} \,dx \right| \\
&= (q-1)\left|\int_{\R^3} \left(1 -\chi_2(\eps x)\right) |\theta_n|^{q-2} z_n \cdot \overline{\left(z_n^+-z_n^-\right)} \,dx \right| \\
& \leq (q-1)\int_{\R^3} \left(1 -\chi_2(\eps x)\right)
|\theta_n|^{q-2} |z_n||z_n^+-z_n^-| \, dx.
\end{align*}
Since, for any $\eps>0$, $\mbox{supp}\, (1 -\chi_2(\eps x)) \subset \R^3$ is
bounded and $z_n \to 0$ in $L^q_{\textnormal{loc}}(\R^3,
\C^4)$ as $n \to \infty$, by H\"older's inequality, we then
obtain that
$$
\left|\int_{\R^3} \left(1 -\chi_2(\eps x)\right) \left(|u_n|^{q-2}
u_n-|u|^{q-2}u \right) \cdot \overline{\left(z_n^+-z_n^-\right)}
\,dx \right|=o_n(1).
$$
Applying again the mean value theorem, we deduce from Lemma
\ref{mprop} that
\begin{align*}
&\left|\int_{\R^3} \left(1 -\chi_2(\eps x)\right) \left(|u_n|^{q-2}
(m_{\eps}(|u_n|^2))^{\frac{3-q}
{2}}u_n-|u|^{q-2}(m_{\eps}(|u|^2))^{\frac{3-q}{2}}u \right) \cdot \overline{\left(z_n^+-z_n^-\right)} \,dx \right| \\
& \leq  (q-1)  \left |\int_{\R^3} \left(1 -\chi_2(\eps x)\right) |\theta_n|^{q-2}(m_{\eps}(|\theta_n|^2))^{\frac{3-q}{2}} z_n \cdot \overline{\left(z_n^+-z_n^-\right)}\, dx \right| \\
& \quad + (3-q)  \left|\int_{\R^3}  \left(1 -\chi_2(\eps x)\right) |\theta_n|^{q-2} (m_{\eps}(|\theta_n|^2))^{\frac{3-q}{2} -1} b_{\eps}(|\theta_n|^2) |\theta_n|^2 z_n \cdot \overline{\left(z_n^+-z_n^-\right)}\, dx \right| \\
& \leq 2  \int_{\R^3} \left(1 -\chi_2(\eps x)\right) |\theta_n|^{q-2}(m_{\eps}(|\theta_n|^2))^{\frac{3-q}{2}} |z_n ||z_n^+-z_n^-|\, dx \\
& \leq \frac{2c}{\eps^{\frac{3-q}{2}}}\int_{\R^3} \left(1
-\chi_2(\eps x)\right) |\theta_n|^{q-2}|z_n ||z_n^+-z_n^-|\, dx,
\end{align*}
from which we also know that, for any $\eps>0$,
\begin{align*}
&\left|\int_{\R^3} \left(1 -\chi_2(\eps x)\right) \left(|u_n|^{q-2}
(m_{\eps}(|u_n|^2))^{\frac{3-q}
{2}}u_n-|u|^{q-2}(m_{\eps}(|u|^2))^{\frac{3-q}{2}}u \right) \cdot \overline{\left(z_n^+-z_n^-\right)} \,dx \right|=o_n(1).
\end{align*}
Similarly, we can derive that
\begin{align*}
& \left|\int_{\R^3}  \left(1 -\chi_2(\eps x)\right) \left(
|u_n|^{q} (m_{\eps}(|u_n|^2))^{\frac{3-q}
{2}-1} b_{\eps}(|u_n|^2)u_n -|u|^{q}(m_{\eps}(|u|^2))^{\frac{3-q}{2}-1} b_{\eps}(|u|^2)u \right) \right.  \\
&\hspace{7.5em} \cdot \overline{\left(z_n^+-z_n^-\right)}
\,dx \bigg| =o_n(1).
\end{align*}
Therefore,
\begin{align} \label{term1}
\left|\int_{\R^3} \left(1 -\chi_2(\eps x)\right)
\left(h_{\eps}(|u_n|) u_n-h(|u|)u \right) \cdot
\overline{\left(z_n^+-z_n^-\right)} \,dx \right|=o_n(1).
\end{align}
We now turn to estimate the second term in the left hand side of
\eqref{sconv}. Reasoning as \eqref{jjjnhyytggff}, we can show that
\begin{align*}
& \int_{\R^3} |V(\eps x)| \left(1 - \frac 18 \chi_1(\eps x) \tilde{\xi}(x ,|u_n|)\right)| z_n| \cdot |z_n^+-z_n^-| \, dx \\
&\leq a\min\left\{\frac 7 8, \theta\right\}\int_{\Omega_{n, 2}} |z_n|
\cdot |z_n^+-z_n^-|\, dx + 2a\int_{\R^3 \backslash\Omega_{n,
2}}\phi \cdot |z_n^+-z_n^-|\, dx.
\end{align*}
Recall that $z_n \wto 0$ in $E$ as $n \to \infty$, then $z_n (x)
\to 0$  and $z^+_n-z^-_n\rightarrow 0$ in $L^2_{loc}(\R^2, \C^2)$.
In addition, by the definition of $\Omega_{n, 2}$, we know that
$|z_n(x)|^2 \leq |u(x)|^2 + 9 |\phi(x)|^2$ for any $x \in \R^3
\backslash \Omega_{n, 2}$. It then follows from the Lebesgue
dominated convergence theorem that
$$
 \int_{\Omega_{n, 2}} |z_n| \cdot |z_n^+-z_n^-|\, dx
\leq c \int_{\R^3 \backslash \Omega_{n, 2}} |z_n|^2 \, dx=o_n(1).
$$
And by  $z^+_n-z^-_n\rightarrow 0$ in $L^2_{loc}(\R^2, \C^2)$, we
get that
$$\int_{\R^3}\phi \cdot |z_n^+-z_n^-|\, dx=o_n(1).$$ This jointly with
\eqref{sconv} and \eqref{term1} gives that
$$
 \|z_n\|^2 \leq o_n(1).
$$
Hence we derive that $u_n \to u$ in $E$ as $n \to \infty$, and the proof is completed.
\end{proof}

In what follows, we shall check that $\Gamma_{\eps}$ satisfies the remaining conditions in Theorem \ref{theorem}.
By Lemma \ref{mprop} and the definition of $h_{\eps}$, it is
clear that there exists a constant $c>0$ such that
$$
h_{\eps}(t) t \leq \frac 18 \min\left\{1-\theta, \frac 18 \right\} t + c t^2 \quad
\mbox{for any} \,\, t \geq 0.
$$
Additionally, it follows from \cite[Lemma 3.1]{WZ1} that
$$
\left|\int_{\R^3} V(\eps x) |u|^2 dx -2 Q_{\eps}(u)\right| \leq \max\left\{\theta, \frac 7 8\right\} \|u\|^2 + c \eps^5.
$$
Therefore, by the definition of $g_{\eps}$, we obtain that
$$
\Gamma_{\eps}(u) \geq \frac 14 \min\left\{1-\theta, \frac 18 \right\} \|u\|^2 -c \|u\|^3-c \eps^5 \quad \mbox{for any}\,\, u \in E^+.
$$
This then yields the following statement.

\begin{lem} \label{existi}
There are $r_0>0$ and $\rho>0$ such that
\begin{align*}
\inf_{u \in E^+, \, \|u\|=r_0} \Gamma_{\eps}(u) \geq \rho \quad
\mbox{for any} \, \, \eps>0 \,\, \mbox{small enough}.
\end{align*}
\end{lem}

Let $\{e_n\} \subset E^+$ be an orthogonal basis and
$E_N=\mbox{span}\{e_1, \dots, e_N\}$. Since $h_{\eps}(t) \geq t^{q-2}$ for any $t \geq 0$,
arguing as the proof of \cite[Lemma 3.2]{WZ}, then we easily get the following result.

\begin{lem} \label{existii}
For any $N \in \N^+$, there exist $\eps_N>0$ and $R_N>0$ with
$R_N>r_0$ such that, for any $0<\eps< \eps_N$,
$$
\sup_{ u \in E^- \oplus E_N, \, \|u\| \geq R_N} \Gamma_{\eps}(u)
\leq 0.
$$
In particular,
\begin{align} \label{ebdd}
\sup_{ u \in E^- \oplus E_N} \Gamma_{\eps}(u)  \leq 2 R_N^2.
\end{align}
\end{lem}

Let $\mathcal{S}$ be a countable dense subset of the dual space of
$E^-$ and
$$
\mathcal{D}=\{d_s: s \in \mathcal{S}, \, d_s(u, v)=|s(u-v)|, \, u,
v \in E^-\}
$$
be a family of semi-metrics on $E^-$. Let $\mathcal{P}$ be a
family of semi-norms on $E$ consisting of all semi-norms
$$
p_s: E \to \R, \quad p_s(u)=|s(u^-)| + \|u^+\|, \quad u=u^+ + u^-
\in E, s \in \mathcal{S}.
$$
Thus $\mathcal{P}$ induces a product topology on $E$ given by
the $\mathcal{D}$-topology on $E^-$ and the norm topology on
$E^+$. We denote this topology by $(E, \mathcal{T}_{\mathcal{P}})$
and the weak${}^*$ topology on $E^*$ by $(E^*,
\mathcal{T}_{w^*})$. By using some ingredients in the proof of
\cite[Lemma 3.3]{WZ1}, it is not difficult to find that, for any $\eps>0$, $\Gamma_{\eps}$ is sequentially lower semi-continuous
and it is weakly sequentially continuous. Therefore, by
\cite[Proposition 4.1]{BD}, we can get
the following.

\begin{lem}
For any $\eps>0$, there hold that, for any $c \in \R$,
$\Gamma_{\eps, c}=\{u \in E: \Gamma_{\eps}(u) \geq c\}$ is
$\mathcal{T}_{\mathcal{P}}$-closed and
$\Gamma_{\eps}':(\Gamma_{\eps, c}, \mathcal{T}_{\mathcal{P}}) \to
(X^*, \mathcal{T}_{w^*})$ is continuous.
\end{lem}

Taking into account Theorem \ref{theorem}, we now have the following statement.

\begin{thm} \label{existence} Let $\epsilon_N$ be  determined in Lemma  \ref{existii}.
For any $0<\eps< \eps_N$, $\Gamma_{\eps}$ admits $N$-pairs
critical points $\pm u_{j,\eps}$ with
\begin{align} \label{ebdd1}
\Gamma_{\eps}(u_{j,\eps}) \in \left[\rho, \,\, \sup_{u \in E^-
\oplus E_N} \Gamma_{\eps} (u)\right] \quad \mbox{for any} \,\,
j=1, \dots, N.
\end{align}
\end{thm}

\section{Profile decomposition of solution sequences and exclusion of blowing-up} \label{estimate}

First of all, let us establish boundedness of the critical points $\{u_{j,\eps}\} \subset E$ obtained in Theorem \ref{existence}.

\begin{lem} \label{sbdd}
For any $N \in \N^+$, there exist a constant $\rho_0>0$ and a constant $\eta_N>0$ such
that, for any $0<\eps<\eps_N$,
$$
\rho_0 \leq \|u_{j,\eps}\| \leq \eta_N \quad \mbox{for any} \,\, 1 \leq j \leq N,
$$
where $\eps_N>0$ is determined in Theorem \ref{existence}.
\end{lem}
\begin{proof}
From \eqref{ebdd1}, we can get that there exists a constant $\rho_0>0$ such that $\|u_{j, \eps}\| \geq \rho_0$.
Otherwise, it is easy to show that $\Gamma_{\eps}(u_{j,\eps})=o_{\eps}(1)$. This is impossible.
Using \eqref{ebdd}, \eqref{ebdd1} and following the ideas of the
proof of boundedness of the Palais-Smale sequence of Lemma
\ref{ps}, we can obtain that $\|u_{j, \eps}\| \leq \eta_N$. Thus the proof is completed.
\end{proof}

Taking advantage of Lemma \ref{sbdd}, we then have the following profile decomposition with respect to $u_{j, \eps}$.

\begin{lem}\label{cmvnvhhguf7fuufss}
Let $\{\eps_n\} \subset \R^+$ with $\eps_n \to 0$ as $ n \to
\infty$, then, for any $1 \leq j \leq N$, $u_{j ,\eps_n}$ has
a profile decomposition
\begin{align} \label{pd}
u_{j, \eps_n}=\sum_{i \in \Lambda_1} U_{j ,i}(\cdot -x_{j, i, n})
+ \sum_{i \in \Lambda_{\infty}} \sigma_{j, i, n} U_{j,
i}(\sigma_{j, i, n}(\cdot-x_{j, i, n})) + r_n,
\end{align}
where $\{x_{j ,i ,n}\} \subset \R^3$ for any $i \in \Lambda_{1} \cup \Lambda_{\infty}$, $\{\sigma_{j, i, n} \} \subset \R^+$
satisfies $\lim_{n \to \infty} \sigma_{j ,i, n}
=\infty$ for any $i \in \Lambda_{\infty}$, $\Lambda_{1}, \Lambda_{\infty}$ are finite sets and the
numbers of the sets are bounded from above by an integer depending
only on $N$. Moreover,
\begin{enumerate}
\item[$(1)$] for any $i \in \Lambda_1$, $u_{j, \eps_n}(\cdot +
x_{j, i, n}) \wto U_{j, i} \neq 0 $ in $H^{1/2}(\R^3, \C^4)$ as $n
\to \infty$ and for any $i \in \Lambda_{\infty}$,
$$
\sigma_{j, i, n}^{-1}u_{j, \eps_n}(\sigma_{j, i, n}^{-1}\cdot +
x_{j, i, n}) \wto U_{j, i} \neq 0 \,\, \mbox{in} \,\,
\dot{H}^{1/2}(\R^3, \C^4) \,\, \mbox{as}\,\,  n \to \infty,
$$
where the homogeneous Sobolev space $\dot{H}^{1/2}(\R^3, \C^4)$ is
defined by
$$
\dot{H}^{1/2}(\R^3, \C^4):=\left\{u \in L^3(\R^3, \C^4) : (-\Delta
)^{1/4} u \in  L^2(\R^3, \C^4)\right\}
$$
with the inner product $(u ,v) =((-\Delta)^{1/4} u, (-\Delta)^{1/4} v)_2$
and the norm $\|u\|_{\dot{H}^{1/2}}^2=(u, u)$ for any $u, v \in \dot{H}^{1/2}(\R^3, \C^4)$.
\item[$(2)$] For any $i, i' \in \Lambda_1 $ with $i \neq i'$,
$|x_{j, i, n}-x_{j, i', n}| \to \infty$ as $n \to \infty$ and for
any $i, i' \in \Lambda_{\infty}$ with $i \neq i'$,
$$
\frac{\sigma_{j, i, n}}{\sigma_{j, i', n}}+\frac{\sigma_{j, i',
n}}{\sigma_{j, i, n}}+ \sigma_{j, i, n}\sigma_{j, i', n}|x_{j, i,
n}-x_{j, i', n}|^2 \to \infty \,\, \mbox{as} \,\, n \to \infty.
$$
\item[$(3)$] There holds that
$$
\sum_{i \in \Lambda_1 \cup \Lambda_{\infty}} \int_{\R^3} |U_{j
,i}|^3 \, dx \leq \liminf_{n \to \infty} \int_{\R^3} |u_{j, \eps_n}|^3 \, dx.
$$
\item[$(4)$] $r_n \to 0$ in $L^3(\R^3, \C^4)$ as $n \to \infty$.
\end{enumerate}
\end{lem}
\begin{proof}
For the proof of this lemma, we refer the readers to the proof of \cite[Lemma 4.2]{CG}, which indeed benefits from \cite[Theorem 3.3]{SST}.
\end{proof}

In the following, our principal aim is to demonstrate that the critical points $\{u_{j, \eps_n}\} \subset E$
obtained in Theorem \ref{existence} cannot blow up, i.e. $\Lambda_{\infty}=\emptyset$. More precisely, we will establish the following crucial result.

\begin{prop}\label{bcbvhfyfyufuadx}
Let $\{u_{j,\epsilon}\}$ be the  critical points obtained in
Theorem \ref{existence}. There exist $M_N>0$ and $\epsilon''_N>0$
such that, for any $0<\epsilon<\epsilon''_N$,
\begin{align}\label{nvcbv99vifkjjuuaaz}
\sup_{x\in\mathbb{R}^3}|u_{j,\epsilon}(x)|<M_N \quad \mbox{for
any} \,\, 1\leq j\leq N.
\end{align}
\end{prop}

To prove this proposition, we need the following lemmas.

\begin{lem}\label{xncbdggdtdtttd}
If $i\in\Lambda_\infty$, then, up to a subsequence,
$x_{j,i}:=\lim_{n\rightarrow\infty}\epsilon_nx_{j,i,n}\in
\mathcal{M}^{\delta_0}$ and $U_{j,i}$ satisfies the equation
\begin{align}\label{cnvcbv99fufjhhf}
-\textnormal{i}  \alpha \cdot \nabla U_{j,i}
=(1-\chi(x_{j,i}))|U_{j,i}| U_{j,i}. 
\end{align}
where $U_{j,i}$ is given in Lemma \ref{cmvnvhhguf7fuufss}.
\end{lem}
\begin{proof}
For simplicity, we shall denote $x_{j,i,n}$, $\sigma_{j,i,n}$ and $u_{j, \eps_n}$ by $x_n$,
$\sigma_n$ and $u_n$, respectively. Let us define $\omega_n=\sigma^{-1}_nu_n(\sigma^{-1}_n\cdot+x_n)$.
Since $u_n$ satisfies \eqref{mdirac} with $\eps=\eps_n$, then $w_{n}$ solves the equation
\begin{align}\label{1bvnghy77g7yfftt}
\begin{split}
&-\textnormal{i} \alpha \cdot \nabla \omega_n + \sigma^{-1}_n a  \beta \omega_n + \sigma^{-1}_n V(\epsilon_n (\sigma^{-1}_nx+x_n)) \omega_n + \sigma_{n}^{-1}\mathcal{Q}_{n, 1}\omega_n \\
&=\sigma^{-1}_n f_{\epsilon_n}(\sigma^{-1}_nx+x_n,
\sigma_n|\omega_n|)\omega_n,
\end{split}
\end{align}
where
\begin{align*}
\mathcal{Q}_{n, 1}(x):=\frac 1 8 \chi_{1}(\epsilon_n (\sigma^{-1}_nx+x_n)) V(\epsilon_n (\sigma^{-1}_nx+x_n)) \tilde{\xi}(\sigma^{-1}_nx+x_n, \sigma_n |w_n|).
\end{align*}
Since $\sigma_n\rightarrow\infty$ as $n \to \infty$, see Lemma \ref{cmvnvhhguf7fuufss}, from Lemma
\ref{sbdd} and H\"older's inequality, we then get that, for any
$\varphi\in C^\infty_0(\mathbb{R}^3, \C^4 )$,
\begin{align}\label{cnvbfggftdhduu88uu}
\sigma^{-1}_n\int_{\mathbb{R}^3}a  \beta \omega_n \cdot
\overline{\varphi} \, dx \rightarrow 0, \quad
\sigma^{-1}_n\int_{\mathbb{R}^3}V(\epsilon_n (\sigma^{-1}_nx+x_n))
\omega_n \cdot \overline{\varphi}  \, dx \rightarrow 0
\end{align}
and
\begin{align} \label{qterm}
\sigma^{-1}_n\int_{\mathbb{R}^3}\mathcal{Q}_{n,1}\omega_n \cdot \overline{\varphi}  \, dx \rightarrow 0 \,\,\mbox{as}\,\, n\rightarrow\infty
\end{align}
where we used the fact that $0 \leq \tilde{\xi}(x, t) \leq 2$ for any $x \in \R^3$ and $t \in \R$.
In addition, by the definition of $g_{\eps_n}$, there holds that
\begin{align}\label{cvmbngjgut77t77t}
\sigma^{-1}_n\int_{\mathbb{R}^3}\chi_2(\epsilon_n
(\sigma^{-1}_nx+x_n))g_{\eps_n}(\epsilon_n (\sigma^{-1}_nx+x_n),
\sigma_n|\omega_n|)\omega_n \cdot \overline{\varphi}  \, dx
\rightarrow 0 \,\,\mbox{as}\,\, n\rightarrow\infty.
\end{align}

We argue by contradiction that $x_{j, i}\not\in\mathcal{M}^{\delta_0}$ for some $ 1 \leq j
\leq N$. From the definition of $\chi_2$, we then have that, for any
$R>0$,
\begin{align}\label{cnvbgfhhfyfhhfyfggfdrf}
\sup_{|x|<R}\left|1-\chi_2(\epsilon_n
(\sigma^{-1}_nx+x_n))\right|\rightarrow 0 \,\, \mbox{as} \,\,
n\rightarrow\infty.
\end{align}
By Lemma \ref{mprop} and the definition of $h_{\eps_n}$, we see that
$$
\sigma^{-1}_n h_{\epsilon_n}(\sigma_n|\omega_n|)\omega_n
\leq\sigma^{q-3}_n|\omega_n|^{q-1}+|\omega_n|^2.
$$
Thus, from Lemma \ref{sbdd} and \eqref{cnvbgfhhfyfhhfyfggfdrf}, we
infer that, for any $\varphi\in C^\infty_0(\mathbb{R}^3, \C^4)$,
\begin{align}\label{2wcvmbngjgut77t77t}
\sigma^{-1}_n\int_{\mathbb{R}^3}(1-\chi_2(\epsilon_n(\sigma^{-1}_nx+x_n)))h_{\epsilon_n}(\sigma_n
|\omega_n|)\omega_n\cdot \overline{\varphi} \, dx \rightarrow 0
\,\, \mbox{as} \,\, n\rightarrow\infty.
\end{align}
Note that $\omega_n \wto U_{j, i}$ in $\dot{H}^{1/2}(\R^3, \C^4)$
as $n \to \infty$, see Lemma \ref{cmvnvhhguf7fuufss}. Applying
\eqref{1bvnghy77g7yfftt}, \eqref{cnvbfggftdhduu88uu},
\eqref{cvmbngjgut77t77t}, \eqref{qterm} and
\eqref{2wcvmbngjgut77t77t}, we now derive that
\begin{align}\label{cnvbvnghfhyf777f7f}
-\textnormal{i} \alpha \cdot \nabla U_{j,i}=0.
\end{align}
It follows that $U_{j,i}=0.$ This contradicts Lemma \ref{cmvnvhhguf7fuufss}. Therefore, we have that
$x_{j,i}\in\mathcal{M}^{\delta_0}.$

Using again  Lemma \ref{mprop} and the definition of $h_{\eps_n}$,
we obtain that $h_{\eps_n}(t) \to t^{q-2} + t$ as $n \to \infty$.
Observe that $\omega_n \to U_{j, i}$ a.e.  in $\R^3$ as $n \to
\infty$, then
$$
\sigma^{-1}_n h_{\epsilon_n}(\sigma_n|\omega_n|)\omega_n \rightarrow
|U_{j,i}|U_{j,i} \quad  \mbox{a.e} \,\, \mbox{in} \,\, \R^3 \,\,
\mbox{as} \,\, n \to \infty.
$$
Moreover, we see that, for any $R>0$,
\begin{align*}
\sup_{|x|<R}\left|\chi_2(\epsilon_n(\sigma^{-1}_n x+x_n))-\chi_2(x_{j,i})\right|\rightarrow
0 \,\, \mbox{as} \,\, n\rightarrow\infty.
\end{align*}
Consequently, it holds that
\begin{align}\label{5wcvmbngjgut77t77t}
\begin{split}
&\sigma^{-1}_n\int_{\mathbb{R}^3}(1-\chi_2(\epsilon_n (\sigma^{-1}_n
x+x_n)))h_{\epsilon_n}(
|\sigma_n\omega_n|)\omega_n\cdot \overline{\varphi}  \, dx\\
&\rightarrow \int_{\mathbb{R}^3}(1-\chi_2(x_{j,i})
|U_{j,i}|U_{j,i}\cdot\overline{\varphi} \, dx \,\, \mbox{as} \,\,
n\rightarrow\infty.
\end{split}
\end{align}
By using \eqref{1bvnghy77g7yfftt}, \eqref{cnvbfggftdhduu88uu}, \eqref{qterm},
\eqref{cvmbngjgut77t77t} and \eqref{5wcvmbngjgut77t77t}, we then
get that $U_{j, i}$ satisfies \eqref{cnvcbv99fufjhhf}. Thus the proof is completed.
\end{proof}

\begin{lem}\label{3cncbvggftd6ettd66d}
If $i\in\Lambda_\infty$, then there exists a constant $C>0$ such
that
$$
|U_{j,i}(x)|\leq \frac{C}{1+|x|^2} \quad \mbox{for any}\,\, x\in\mathbb{R}^3.
$$
\end{lem}
\begin{proof}
With Lemma \ref{xncbdggdtdtttd} in hand, the result of this lemma is a straightforward consequence of \cite[Theorem 1.1]{bo}.
\end{proof}

\begin{lem}\label{2xncbdggdtdtttd}
If $i\in\Lambda_1$, then, up to a subsequence,
$x_{j,i}:=\lim_{n\rightarrow\infty}\epsilon_nx_{j,i,n}\in
\mathcal{M}^{\delta_0}$ and $U_{j,i}$ satisfies the equation
\begin{align} \label{sxdd4mdirac}
\hspace{-0.75cm}-\textnormal{i} \alpha \cdot \nabla U_{j,i} + a
\beta U_{j,i} +V(x_{j,i}) U_{j,i}=(1-\chi_2(x_{j,i})) \left(
|U_{j,i}|^{q-2}+ |U_{j,i}|^2\right)U_{j,i}.
\end{align}
\end{lem}
\begin{proof}
For simplicity, we shall write $u_n=u_{j, \eps_n}$ and $x_n=x_{j,i,n}$.
Let us define $w_n:=u_n(\cdot + x_n)$. Since $u_n$ satisfies \eqref{mdirac} with $\eps=\eps_n$, then
\begin{align} \label{equwn1}
-\textnormal{i} \alpha \cdot \nabla w_n + a  \beta
w_n + V(\epsilon_n (x + x_n))
w_n - \mathcal{Q}_{n, 2} w_n=f_{\epsilon_n}(x +  x_n, |w_n|)w_n,
\end{align}
where $\mathcal{Q}_{n, 2}$ is defined by
$$
\mathcal{Q}_{n, 2}(x):=\frac 1 8 \chi_{1}(\epsilon_n (x+x_n)) V(\epsilon_n (x+x_n)) \tilde{\xi}(x+x_n, |w_n|).
$$
Recall that $0 \in \mathcal{M}$, according to the definitions of
$g_{\eps_n}$ and $\chi_2$, it holds that,  for any $\varphi\in
C^\infty_0(\R^3,\C^4)$,
\begin{align} \label{glimit}
\begin{split}
&\left|\int_{\R^3} \chi_2(\eps_n(x+ x_n))g_{\eps_n}(x+ x_n, |w_n|)
w_n\cdot \overline{\varphi}\, dx\right| \\
&\leq\int_{\R^3}
\chi_2(\eps_n(x+ x_n))\phi(x+
x_n)|w_n||\varphi| \, dx\\
&=o_n(1).
\end{split}
\end{align}

We assume by contradiction that $x_{j,i} \not \in
\mathcal{M}^{\delta_0}$ for some $1 \leq j \leq N$. By the definition of $\chi_2$ and $w_n \wto U_{j ,i} \neq 0$ in
$E$ as $n \to \infty$, we then get that, for any $\varphi\in
C^\infty_0(\R^3,\C^4)$,
\begin{align}\label{cnvbhhfgrt66ftrrr}
\lim_{n\rightarrow\infty}\int_{\R^3}(1-\chi_2(\epsilon_n(x+x_n)))|w_n|w_n\cdot\overline{\varphi}\,dx=0.
\end{align}
From\eqref{glimit} and \eqref{cnvbhhfgrt66ftrrr}, we then obtain that, for any $\varphi\in C^\infty_0(\R^3,\C^4)$,
\begin{align}\label{cnvbghfyyf6f66f}
\lim_{n\rightarrow\infty}\int_{\R^3}f_{\epsilon_n}(x +  x_n,
|w_n|)w_n\cdot\overline{\varphi}\,dx=0.
\end{align}
Multiplying both sides of \eqref{equwn1} by $U^+_{j,i}-U^-_{j,i}$
and integrating in $\R^3$ and utilzing \eqref{cnvbghfyyf6f66f}, we have that
\begin{align*} 
o_n(1) &=\langle w^+_n-w^-_n, U^+_{j,i}-U^-_{j,i}\rangle
+\int_{\R^3}\left(V(\epsilon_n(x+x_n))-\mathcal{Q}_{n,2}(x)\right)w_n\cdot
\overline{U^+_{j,i}-U^-_{j,i}}\,dx\nonumber\\
&\geq\langle w^+_n-w^-_n, U^+_{j,i}-U^-_{j,i}\rangle
-\int_{\R^3}\left|V(\epsilon_n(x+x_n))-\mathcal{Q}_{n,2}(x)\right||w_n|
|U^+_{j,i}-U^-_{j,i}|\,dx,
\end{align*}
where $\langle \cdot,  \cdot \rangle$ denotes the inner product in $E$. Note that $|V(x)|\leq a$ for any $x \in \R^3$, $0 \leq \tilde{\xi}(x, t) \leq 2$ for
any $x \in \R^3, t \in \R$ and $0 \leq \chi_1(x) \leq 1$ for
any $x \in \R^3$. Therefore,
\begin{align}\label{23cnv89f8fr6ryfhhf}
o_n(1) \geq\langle w^+_n-w^-_n,
U^+_{j,i}-U^-_{j,i}\rangle-a\int_{\R^3}|w_n|
|U^+_{j,i}-U^-_{j,i}|\,dx.
\end{align}
Recall that $ w_n \wto U_{j ,i} \neq 0$ in $E$ as $n \to \infty$,
see Lemma \ref{cmvnvhhguf7fuufss}, it then holds that $|w_n| \wto
|U_{j ,i}|$ in $L^2(\R^3)$ as $n \to \infty$ and
\begin{align}\label{1cnvbfhhfyfgggd}
\langle w^+_n-w^-_n, U^+_{j,i}-U^-_{j,i}\rangle=
\|U_{j,i}\|^2+o_n(1).
\end{align}
In addition,
\begin{align}\label{qq2cnvbfhhfyfgggd}
\hspace{-1cm}\int_{\R^3}|w_n| |U^+_{j,i}-U^-_{j,i}|\,dx&=\int_{\R^3}|U_{j,i}| |U^+_{j,i}-U^-_{j,i}|\,dx+o_n(1)
\leq\|U_{j,i}\|^2_{L^2(\R^3)}+o_n(1).
\end{align}
Therefore, combining \eqref{l2} and
\eqref{23cnv89f8fr6ryfhhf}-\eqref{qq2cnvbfhhfyfgggd} yields that
$$ \|U_{j, i}\|^2 -a \| U_{j, i}\|_{L^2(\R^3)}^2\leq0,
$$
from which we have that $U_{j, i}=0$, because of $\sigma(H_a)=\R
\backslash (-a, a)$.
This is impossible.
Hence we get that $x_{j, i} \in \mathcal{M}^{\delta_0}$.

We now prove that $U_{j,i}$ satisfies \eqref{equwn1}. Thanks to $x_{j,i}\in\mathcal{M}^{\delta_0}$,
from the definition of $\chi_1$,
we have that, for any $R>0,$
$$
\sup_{|x|<R}\left|\chi_{1}(\epsilon_n (x+x_n)) \right|=o_n(1).
$$
It then follows from the definition of $\mathcal{Q}_{n,2}$ that,
for any $\varphi\in C^\infty_0(\R^3,\C^4),$
\begin{align}\label{cnvbuut7yys444w}
\int_{\R^3}\mathcal{Q}_{n, 2}w_n\cdot\overline{\varphi}\,dx=o_n(1).
\end{align}
Since $x_{j,i}\in\mathcal{M}^{\delta_0}$ and
$ w_n \wto U_{j ,i} $ in $E$ as $n \to \infty$, it is not hard to deduce that for any
$\varphi\in C^\infty_0(\R^3, \C^4)$,
\begin{align}\label{ncbvhhfy6d77d6d6}
\int_{\R^3} V(\epsilon_n (x + x_n))
w_n \cdot\overline{\varphi}\,dx=\int_{\R^3} V(x_{j,i})U_{j,i}\cdot\overline{\varphi}\,dx+o_n(1)
\end{align}
and
\begin{align}\label{2ncbvhhfy6d77d6d6}
\hspace{-1.5cm}\int_{\R^3} f_{\epsilon_n}(x + x_n, |w_n|)
w_n \cdot\overline{\varphi}\,dx=\int_{\R^3} (1-\chi_2(x_{j,i})(|U_{j,i}|^{q-2}+|U_{j,i}|)U_{j,i}
\cdot\overline{\varphi}\,dx+o_n(1),
\end{align}
where we used \eqref{glimit}. In view of \eqref{equwn1} and \eqref{cnvbuut7yys444w}-\eqref{2ncbvhhfy6d77d6d6},
we then derive that $U_{j, i}$ satisfies \eqref{sxdd4mdirac}, and the proof is completed.
\end{proof}

\begin{lem}\label{7cncbvggftd6ettd66d}
If $i\in\Lambda_1$, then there exist $c, C>0$ such that
$$
|U_{j,i}(x)|\leq C\, \textnormal{exp} \left(-c|x|\right) \quad \mbox{for any}\,\,
x\in\mathbb{R}^3.
$$
\end{lem}
\begin{proof}
The proof of this lemma can be completed by closely following the ideas of the proof of \cite[Lemma
4.6]{WZ}, so we omit its proof here.
\end{proof}

Let $i_\infty\in\Lambda_\infty$ be such that
\begin{align}\label{mcnvbbbvuvjyyfffff}
\sigma_n:=\sigma_{j,i_\infty,n}=\min\{\sigma_{j,i,n} \ |\
i\in\Lambda_\infty\}.
\end{align} Denote
\begin{align}\label{mmmzcczdr4s4s4sss}
x_n:=x_{j,i_\infty,n}.
\end{align}

\begin{lem}\label{ncbvjgfu88g8g8f7}
There exists a constant $\overline{C}>0$ such that, up to a
subsequence, the set
$$
\mathcal{A}^1_{n}=B_{(\overline{C}+5)\sigma^{-\frac{1}{2}}_n}(x_n)\setminus
B_{\overline{C}\sigma^{-\frac{1}{2}}_n}(x_n)
$$  satisfies
$$
\mathcal{A}^1_{n}\cap \{x_{j,i,n} \ |\
i\in\Lambda_\infty\}=\emptyset.
$$
\end{lem}
\begin{proof}
Arguing as the proof of \cite[Lemma 4.8]{CLW}, one can obtain the result of this lemma with some minor changes.
\end{proof}

Let us now introduce an inequality with respect to a massless Dirac operator in $L^p(\R^3, \C^4)$
for $1<p<\infty$, which will be employed frequently in our proofs.

\begin{lem} \label{di} \cite[Lemma 4.2]{IS}
For $1<p< \infty$, there exists a constant $C>0$ such that
$$
\|\nabla u\|_{L^p{(\R^3)}} \leq C \|(-\textnormal{i} \alpha \cdot \nabla) u\|_{L^p(\R^3)}
$$
for every $u \in C_0^{\infty}(\R^3, \C^4)$.
\end{lem}

We next give a crucial lemma used to establish estimates of the solutions on safe domains.

\begin{lem}\label{cnvbghgyf7fyfyf66t}
There exists a constant $\gamma_*>0$ such that if $M_n\in L^3(\mathbb{R}^3, \R)$ satisfies
\begin{align*}
\|M_n\|_{L^3(\R^3)}\leq\gamma_*,
\end{align*}
then, for any $h_n\in L^2(\R^3,\C^4)$, the equation
\begin{align*}
-\textnormal{i} \alpha \cdot \nabla w_n -M_n w_n=h_n
\end{align*}
has a unique solution $w_n\in H^1(\R^3, \C^4).$ Moreover,
\begin{enumerate}
\item[$(\textnormal{i})$] if $h_n$ satisfies one of following
conditions,
\begin{align*}
&(1)\,\,\, |h_n(x)|\leq C \sigma_{j,i,n}\sigma_{j,i',n}|U_{j,i}(\sigma_{j,i,n}(x-x_{j,i,n}))| |U_{j,i'}(\sigma_{j,i',n}(x-x_{j,i',n}))|, \\
&(2) \,\,\, |h_n(x)|\leq C \sigma_{j,i,n}|U_{j,i}(\sigma_{j,i,n}(x-x_{j,i,n}))| \sum_{i'\in\Lambda_\infty} \sigma_{j,i',n} |U_{j,i'}(\sigma_{j,i',n}(x-x_{j,i',n}))|, \\
&(3)\,\,\, |h_n(x)|\leq C \sigma_{j,i,n}|U_{j,i,n}(\sigma_{j,i,n}(x-x_{j,i,n}))|\left(|M_n(x)|+|r_n(x)|\right),
\end{align*}
then there exists a constant $C_p>0$ independent of $n$ such that
$\|w_n\|_{L^p(\mathcal{A}^2_n)}\leq C_p$ for any $p>3$, where
$i,i'\in\Lambda_\infty,$ $r_n$ is given by \eqref{pd} and
\begin{align} \label{defan2}
\mathcal{A}^2_{n}:=B_{(\overline{C}+9/2)\sigma^{-\frac{1}{2}}_n}(x_n)\setminus
B_{(\overline{C}+1/2)\sigma^{-\frac{1}{2}}_n}(x_n).
\end{align}
\item[$(\textnormal{ii})$] If $h_n$ satisfies that,  there exist
$i\in\Lambda_\infty$ and $i'\in\Lambda_1,$ such that
\begin{align}
\qquad \,\,\, |h_n(x)| &\leq
C\sigma_{j,i,n}^{q-2}|U_{j,i}(\sigma_{j,i,n}(x-x_{j,i,n}))|^{q-2}\left(|U_{j,i'}(x-x_{j,i',n})|^{q-2}+|U_{j,i'}(x-x_{j,i',n})|\right)\nonumber\\
&\quad+\sigma_{j,i,n}|U_{j,i}(\sigma_{j,i,n}(x-x_{j,i,n}))|\left(|U_{j,i'}(x-x_{j,i',n})|^{q-2}+|U_{j,i'}(x-x_{j,i',n})|\right),\nonumber
\end{align}
then there exists a constant $C_p>0$ independent of $n$ such that $\|w_n\|_{L^{p}(\R^3)}\leq C_p $ for any $p \geq 3/2.$
\item[$(\textnormal{iii})$] If $h_n$ satisfies  $\|h_n\|_{L^2(\R^3)}+\|h_n\|_{L^3(\R^3)} \leq C$,
then there exists a constant $C_p>0$ independent of $n$ such that $ \| w_n\|_{L^p(\R^3)}\leq C_p$ for any $ p \geq 2$.
\end{enumerate}
\end{lem}
\begin{proof}
For the proof of this lemma, we refer the readers to the proofs of
\cite [Lemmas 4.9-4.15]{CG}.
\end{proof}

For simplicity, we shall denote $u_{j, \eps_n}$ by $u_n$ in the following.

\begin{lem}\label{mcnvjjf0000d0eq3w3}
Assume $5/2<q<3$, then $r_n\in L^p(\R^3,\C^4)$ and
$\lim_{n\rightarrow\infty}\|r_n\|_{L^p(\R^3)}=0$ for any $3/2< p
\leq 3$.
\end{lem}
\begin{proof}
We first give the inequality $(4.4)$ of \cite{WZ1}, it reads that
\begin{align} \label{44}
\Delta |u_n| \geq a^2|u_n|-\left(\mathcal{H}_{\eps_n}(x, |u_n|)\right)^2 |u_n|,
\end{align}
where
$$
\mathcal{H}_{\eps_n}(x, |u_n|):=-V(\eps_n x) + \frac 18 \chi_1(\eps_n x) V(\eps_n x) \tilde{\xi}(x, |u_n|) + f_{\eps_n}(x, |u_n|).
$$
Let $W_n=|u_n|$, then there holds that
\begin{align} \label{vcmbnjjhih8ufufttf}
-\Delta W_n+a^2W_n\leq \left(\mathcal{H}_{\eps_n}(x, W_n)\right)^2 W_n, 
\end{align}
If $\eps_n x \notin B_{R_0/2}(0)$, by the definitions of $\chi_2$
and $f_{\eps_n}$, we then find that
\begin{align} \label{rn}
|\mathcal{H}_{\eps_n}(x, W_n)| \leq  |V(\epsilon_nx)| + \phi(\eps_n x),
\end{align}
where $R_0>0$ is the constant appearing in the definition of
$\chi_1$ such that $\mathcal{M}^{\delta_0 +1} \subset
B_{R_0/2}(0)$. By utilizing the condition $(V_1)$, \eqref{rn} and
taking $R_0>0$ larger if necessary, there then holds that
$$
a^2-\left(\mathcal{H}_{\eps_n}(x, W_n)\right)^2 \geq
\frac{a\gamma}{2(1 + |\eps_n x|^{\tau})}.
$$
This along with \eqref{vcmbnjjhih8ufufttf} gives that
\begin{align} \label{equwn}
-\Delta W_n+\frac{a\gamma}{2(1 + |\eps_n x|^{\tau})} W_n\leq  0.
\end{align}
Using the elements presented in the proof of  \cite[Lemma
4.5]{WZ1}, we are now able to get that there are constants $c,
C>0$ such that
\begin{align} \label{decaywn}
W_n(x) \leq C \, \textnormal{exp}\left({-c |x|^{\frac{2-\tau}{2}}}
\right)\quad \mbox{for any} \,\, \eps_n x \in  \R^3 \backslash
B_{R_0/2}(0).
\end{align}
If $\eps_n x \in B_{R_0}(0)$, from the condition $(V_1)$, we then know that there exists $\lambda>0$ such that
\begin{align}\label{cmvnbug7f8f77edfree}
a^2-(1+\lambda)V^2(\epsilon_n x)>\lambda.
\end{align}
Applying the definition of $f_{\epsilon_n}$ and Young's inequality, we obtain
that there exists a constant $C>0$ such that
\begin{align*}
\left(\mathcal{H}_{\eps_n}(x, W_n)\right)^2 W_n\leq
(1+\lambda)V^2(\epsilon_nx)W_n+C \left(W^{2(q-2)}_n+W^2_n
\right)W_n \quad \mbox{for} \ \eps_n x \in B_{R_0}(0).
\end{align*}
According to \eqref{vcmbnjjhih8ufufttf}, we then obtain that
\begin{align} \label{equWn}
-\Delta W_n+\lambda W_n\leq C \left(W^{2(q-2)}_n+W^2_n
\right)W_n \quad \mbox{for} \ \eps_n x \in B_{R_0}(0).
\end{align}
Notice that $5/2<q<3$, then $2<2(q-2)+1<3$. Taking into account
Lemma \ref{sbdd}, \eqref{equwn}, \eqref{decaywn} and \eqref{equWn}, we then have
that
\begin{align*}
\lambda \int_{\mathcal{B}_n} W_n \, dx \leq C+
C\int_{\mathcal{B}_n}(W^{2(q-2)+1}_n + W^3_n) \, dx \leq C,
\end{align*}
where $\mathcal{B}_n:=\{x \in \R^3 : \eps_n x \in B_{R_0}(0)\}$.
This jointly with \eqref{decaywn} shows that
$$
\int_{\R^3} W_n \, dx \leq C.
$$
As a consequence, we know that, for any $1\leq p \leq 3$, there exists a constant $C_p>0$
independent of $n$ such that
\begin{align}\label{vbfggftf6rtt}
\|W_n\|_{L^p(\R^3)}\leq C_p.
\end{align}
By means of Lemmas \ref{3cncbvggftd6ettd66d} and
\ref{7cncbvggftd6ettd66d}, we have that, for any $3/2<p\leq 3$,
there exists a constant $C'_p>0$ independent of $n$ such that
\begin{align}\label{2tfvbfggftf6rtt}
\|\sum_{i \in \Lambda_1} U_{j ,i}(\cdot -x_{j, i, n}) + \sum_{i
\in \Lambda_{\infty}} \sigma_{j, i, n} U_{j, i}(\sigma_{j, i,
n}(\cdot-x_{j, i, n}))\|_{L^p(\R^3)}\leq C'_p.
\end{align}
Using \eqref{pd}, \eqref{vbfggftf6rtt} and
\eqref{2tfvbfggftf6rtt}, we now deduce that, for any $3/2< p\leq
3$, there exists a constant $C_p''>0$ independent of $n$ such that
\begin{align}\label{2vbfggftf6rtt}
\|r_n\|_{L^p(\R^3)}\leq C_p''.
\end{align}
On the other hand, since $r_n \to 0$ in $L^3(\R^3, \C^4)$ as $n \to
\infty$, see Lemma \ref{cmvnvhhguf7fuufss}, it then follows from
\eqref{2vbfggftf6rtt} that $ r_n \to 0$ in $L^p(\R^3, \C^4)$ for any $3/2< p \leq 3$. Thus we
have completed the proof.
\end{proof}

\begin{lem}\label{azscmbnbj88giuguugudd}
Assume $5/2<q<3,$ then, for any  $p \geq 2$, there exists a constant
$C_p>0$ independent of $n$ such that
\begin{align}\label{vcmbniigf8fuufttdf}
\|r_n\|_{L^p(\mathcal{A}^2_n)}\leq C_p,
\end{align}
where $\mathcal{A}_{n}^2$ is defined by \eqref{defan2}.
\end{lem}
\begin{proof}
Let us first consider the case $2 \leq p \leq 3$. In this case, by \eqref{pd}, Lemmas \ref{sbdd},
\ref{3cncbvggftd6ettd66d} and \ref{7cncbvggftd6ettd66d}, it is easy to see that there is a constant $C_p>0$
independent of $n$ such that
$$
\|r_n\|_{L^p(\R^3)}\leq C_p.
$$
This then proves \eqref{vcmbniigf8fuufttdf} for any $2 \leq p \leq 3$. We next consider the
case $p>3$. Note that $u_n$ satisfies the equation
\begin{align*} 
-\textnormal{i} \alpha \cdot \nabla u_n + a  \beta
u_n + V(\epsilon_nx)
u_n - \frac 18 \chi_1(\eps_n x) V(\eps_n x) \tilde{\xi}(x, |u_n|)u_n=f_{\epsilon_n}(x, |u_n|)
u_n.
\end{align*}
From \eqref{pd} and Lemmas \ref{xncbdggdtdtttd} and
\ref{2xncbdggdtdtttd}, we then get that
\begin{align}\label{hggyf77fkalloiu}
\begin{split}
&-\textnormal{i} \alpha \cdot \nabla r_n \\
&= - a \beta r_n- V(\epsilon_nx) r_n +\frac 18 \chi_1(\eps_n x) V(\eps_n x) \tilde{\xi}(x, |u_n|)u_n+ f_{\epsilon_n}(x, |u_n|)u_n \\
&\quad+\sum_{i\in\Lambda_1}\big((\textnormal{i} \alpha \cdot \nabla U_{j,i})(x-x_{j,i,n})-a  \beta U_{j,i}(x-x_{j,i,n})-V(\epsilon_n x)U_{j,i}(x-x_{j,i,n})\big)\\
&\quad+\sum_{i\in\Lambda_\infty}\big(\sigma_{j,i,n}^{2}(\textnormal{i}
\alpha \cdot \nabla
U_{j,i})(\sigma_{j,i,n}(x-x_{j,i,n}))- \sigma_{j,i,n} a\beta U_{j,i}(\sigma_{j,i,n}(x-x_{j,i,n})) \\
&\quad-\sigma_{j,i,n}V(\epsilon_n x) U_{j,i}(\sigma_{j,i,n}(x-x_{j,i,n})) \big)\\
&=- a  \beta r_n - V(\epsilon_nx) r_n +\frac 18 \chi_1(\eps_n x) V(\eps_n x) \tilde{\xi}(x, |u_n|)r_n+f_{\epsilon_n}(x,|u_n|)u_n\\
&\quad +I^1_n(x)+I^\infty_n(x),
\end{split}
\end{align}
where
\begin{align*}
\begin{split}
I^1_n(x)&:=\sum_{i\in\Lambda_1}-\Big((1-\chi_2(x_{j,i})) \left( |U_{j,i}(x-x_{j,i,n})|)^{q-2}+ |U_{j,i}(x-x_{j,i,n})|)^2\right)U_{j,i}(x-x_{j,i,n}) \\
&\quad \qquad \,\, \, \, +V(x_{j,i})U_{j,i}(x-x_{j,i,n})-V(\epsilon_n x)
U_{j,i}(x-x_{j,i,n}) \\
& \quad \qquad \,\, \, \, +\frac 18 \chi_1(\eps_n x) V(\eps_n x)
\tilde{\xi}(x, |u_n|) U_{j,i}(x-x_{j,i,n}) \Big)
\end{split}
\end{align*}
and
\begin{align*}
\begin{split}
I^\infty_n(x)&:=\sum_{i\in\Lambda_\infty}\Big(-\sigma_{j,i,n}^2 (1-\chi(x_{j,i}))|U_{j,i}(\sigma_{j,i,n}(x-x_{j,i,n}))| U_{j,i}(\sigma_{j,i,n}(x-x_{j,i,n})) \\
&\qquad \qquad -\sigma_{j,i,n} a\beta
U_{j,i}(\sigma_{j,i,n}(x-x_{j,i,n}))-\sigma_{j,i,n} V(\epsilon_n
x) U_{j,i}(\sigma_{j,i,n}(x-x_{j,i,n})) \\
&\qquad \qquad +\frac {1}{8} \sigma_{j,i,n}\chi_1(\eps_n x)
V(\eps_n x) \tilde{\xi}(x, |u_n|)
U_{j,i}(\sigma_{j,i,n}(x-x_{j,i,n}))\Big).
\end{split}
\end{align*}
Applying again \eqref{pd}, we have that
\begin{align*}
f_{\epsilon_n}(x,|u_n|)u_n&=M_n(x)r_{n}+\sum_{i\in\Lambda_1}f_{\epsilon_n}(x,|U_{j,i,n}(\cdot-x_{j,i,n})|)r_n\nonumber\\
&\quad+\sum_{i\in\Lambda_\infty}f_{\epsilon_n}(x, \sigma_{j,i,n} |U_{j,i}(\sigma_{j,i,n}(x-x_{j,i,n}))|)r_n\nonumber\\
&\quad+\sum_{i\in\Lambda_1}f_{\epsilon_n}(x, |u_n|)U_{j,i}(x-x_{j,i,n})\nonumber\\
&\quad+\sum_{i\in\Lambda_\infty} \sigma_{j,i,n}
f_{\epsilon_n}(x, |u_n|)U_{j,i}((\sigma_{j,i,n}(x-x_{j,i,n})) \nonumber\\
&\quad-(1-\chi_2(\epsilon_n x))r_n\sum_{i\in\Lambda_\infty}\sigma_{j,i,n}^{q-2}|U_{j,i}(\sigma_{j,i,n}(x-x_{j,i,n}))|^{q-2},
\end{align*}
where
\begin{align*}
M_n(x)&:=f_{\epsilon_n}(x, |u_n|)-\sum_{i\in\Lambda_1}f_{\epsilon_n}(x, |U_{j,i}(x-x_{j,i,n})|) \\ 
&\quad-\sum_{i\in\Lambda_\infty}f_{\epsilon_n}(x, \sigma_{j,i,n}|U_{j,i}(\sigma_{j,i,n}(x-x_{j,i,n}))|)\\ 
&\quad+(1-\chi_2(\epsilon_n
x))\sum_{i\in\Lambda_\infty}\sigma_{j,i,n}^{q-2}|U_{j,i}(\sigma_{j,i,n}(x-x_{j,i,n}))|^{q-2}.
\end{align*}
In view of the definition of $f_{\eps_n}$, it is not hard to find that there exists a constant $C>0$ independent
of $n$  such that, for any $x \in \R^3$,
\begin{align*}
|M_n(x)| \leq C\left(|r_n(x)|+|r_n(x)|^{q-2}
\right)+\sum_{i\in\Lambda_\infty}\sigma^{q-2}_{j,i,n}|U_{j,i}(\sigma_{j,i,n}(x-x_{j,i,n}))|^{q-2}+C\chi_2(\epsilon_n
x)\phi(x).
\end{align*}
Notice that $5/2<q<3,$ then $ 3/2<3(q-2)<3$. By Lemma \ref{mcnvjjf0000d0eq3w3}, we then have that
$$
\lim_{n\rightarrow\infty}\|r_n\|_{L^3(\R^3)}=0, \quad \lim_{n\rightarrow\infty}\|r_n\|_{L^{3(q-2)}(\R^3)}=0.
$$
By Lemma  \ref{3cncbvggftd6ettd66d}, it holds that
$$\lim_{n\rightarrow\infty}\|\sum_{i\in\Lambda_\infty}
|\sigma_{j,i,n}U_{j,i}(\sigma_{j,i,n}(\cdot-x_{j,i,n}))|^{q-2}\|_{L^{3}(\R^3)}=0.$$
Moreover, since $0\in\mathcal{M}$, then $\chi_2(\epsilon_n
x)\rightarrow 0$ a.e. in $\R^3$ as $n\rightarrow\infty.$
Therefore, we get that
$$\lim_{n\rightarrow\infty}\|\chi_2(\epsilon_n\cdot)\phi\|_{L^3(\R^3)}=0.$$ As a consequence, we have that
$\lim_{n\rightarrow\infty}\|M_n\|_{L^3(\R^3)}=0$. This means that,
for any $n\in \N^+$ large enough,
\begin{align}\label{cnvbvhhgfyf66fyf777}
\|M_n\|_{L^3(\R^3)}<\gamma_*,
\end{align}
where $\gamma_*$ is the constant in Lemma
\ref{cnvbghgyf7fyfyf66t}. In addition, by the definition of $M_n$,
we can see that
\begin{align*}
&f_{\epsilon_n}(x,|u_n|)U_{j,i}(x-x_{j,i,n})\nonumber\\
&=\sum_{i'\in\Lambda_\infty}f_{\epsilon_n}(x, \sigma_{j,i',n}|U_{j,i'}(\sigma_{j,i',n}(x-x_{j,i',n}))|)U_{j,i}(x-x_{j,i,n})\nonumber\\
&\quad
-(1-\chi_2(\epsilon_n x))\sum_{i'\in\Lambda_\infty}\sigma_{j,i',n}^{q-2}|U_{j,i'}(\sigma_{j,i',n}(x-x_{j,i',n}))|^{q-2}U_{j,i}(x-x_{j,i,n})\nonumber\\
&\quad+\sum_{i'\in\Lambda_1}f_{\epsilon_n}(x, |U_{j,i'}(x-x_{j,i',n})|)U_{j,i}(x-x_{j,i,n})\nonumber\\
&\quad+M_n(x)U_{j,i}(x-x_{j,i,n})
\end{align*}
 and
\begin{align*}
&\sigma_{j,i,n} f_{\epsilon_n}(x, |u_n|)U_{j,i}(\sigma_{j,i,n}(x-x_{j,i,n}))\nonumber\\
&=\sigma_{j,i,n}
\sum_{i'\in\Lambda_\infty}f_{\epsilon_n}(x, \sigma_{j,i',n} |U_{j,i'}(\sigma_{j,i',n}(x-x_{j,i',n}))|)U_{j,i,n}(\sigma_{j,i,n}(x-x_{j,i,n}))\nonumber\\
&\quad -\sigma_{j,i,n}(1-\chi(\epsilon_n x))\sum_{i'\in\Lambda_\infty}\sigma_{j,i',n}^{q-2}|U_{j,i',n}(\sigma_{j,i',n}(x-x_{j,i',n}))|^{q-2}U_{j,i}(\sigma_{j,i,n}(x-x_{j,i,n}))\nonumber\\
&\quad+\sigma_{j,i,n}\sum_{i'\in\Lambda_1}f_{\epsilon_n}(x, |U_{j,i',n}(x-x_{j,i',n})|)U_{j,i}(\sigma_{j,i,n}(x-x_{j,i,n}))\nonumber\\
&\quad+\sigma_{j,i,n}M_n(x)U_{j,i}(\sigma_{j,i,n}(x-x_{j,i,n})).
\end{align*}
From the discussion above, we know that $r_n$ satisfies the following equation
\begin{align*}
-\textnormal{i} \alpha \cdot \nabla r_n-M_n r_n=\sum_kh_k.
\end{align*}
By means of the definition of $f_{\eps_n}$ and the fact that $0 \leq \tilde{\xi}(x ,t) \leq 2$ for any $x \in \R^3$ and $t \in \R$, it is not difficult to check that every term $h_k$ fulfills one of conditions in Lemma \ref{cnvbghgyf7fyfyf66t} imposed on $h_n$.  Taking advantage of Lemma \ref{cnvbghgyf7fyfyf66t} and \eqref{cnvbvhhgfyf66fyf777}, we then get \eqref{vcmbniigf8fuufttdf}, and the proof is completed.
\end{proof}

\begin{lem}\label{ncbvuuf8f77f7987}
Assume $5/2<q<3,$ then, for any $0<\nu<1/2$, there exists a
constant $C_\nu>0$ independent of $n$ such that
$$
|u_n(x)|\leq C_\nu\sigma^{\nu}_n,\ x\in\mathcal{A}^3_n,
$$
where
$$\mathcal{A}^3_{n}=B_{(\overline{C}+4)\sigma^{-\frac{1}{2}}_n}(x_n)\setminus
B_{(\overline{C}+1)\sigma^{-\frac{1}{2}}_n}(x_n).$$
\end{lem}

\begin{proof}
Note that $0 \leq \tilde{\xi}(x ,t) \leq 2$ for any $x \in \R^3$ and $t \in \R$.
In view of the definition of $I^1_n$ and Lemma \ref{7cncbvggftd6ettd66d},
 there exists a constant $C>0$ independent of $n$ such
that, for any $x \in \R^3$,
\begin{align}\label{nmcnvbfhfyf66fyf}
|I^1_n(x)| \leq C.
\end{align}
By Lemma \ref{ncbvjgfu88g8g8f7}, we know that, for any
$x\in\mathcal{A}^2_n$,
\begin{align*}
\sigma_{j,i,n}|x-x_{j,i,n}|\geq \frac{1}{2}
\sigma_{j,i,n}\sigma^{-\frac{1}{2}}_n\geq\frac{1}{2}\sigma^{\frac{1}{2}}_{j,i,n}.
\end{align*}
It then follows from Lemma \ref{3cncbvggftd6ettd66d} that there
exists a constant $C>0$ independent of $n$ such that, for any
$x\in\mathcal{A}^2_n$,
\begin{align} \label{xcddnvbvh77fyfgttf4e}
|\sigma_{j,i,n} U_{j,i}(\sigma_{j,i,n}(x-x_{j,i,n}))|\leq C.
\end{align}
This then suggests that there exists a constant $C>0$ independent of $n$ such
that, for any $x\in\mathcal{A}^2_n$,
\begin{align}\label{cvngfhfyf6fyfyfyyf}
|I^\infty_n(x)|\leq C.
\end{align}
Let $\eta_{n} \in C_0^{\infty}(\R^3, [0, 1])$ be a cut-off
function satisfying $\eta_n(x)=1$ for any $ x \in
{\mathcal{A}^3_n}$, $ \eta_n(x)=0$ for any $x \not \in
\mathcal{A}^2_n$ and $|\nabla\eta_{n}(x)|\leq C\sigma^{1/2}_n$ for
any $x \in \R^3$, where $C>0$ is independent of $n$, where $\mathcal{A}_{n}^2$ is defined by \eqref{defan2}. In light of \eqref{hggyf77fkalloiu}, then
\begin{align*}
(-\textnormal{i} \alpha \cdot \nabla)(\eta_{n} r_n) &=\eta_{n}(-\textnormal{i} \alpha \cdot \nabla r_n)-\textnormal{i}\sum^3_{k=1} (\partial_k\eta_{n})\alpha_k r_n \nonumber  \\
&=\eta_{n}( - a \beta r_n - V(\epsilon_n x) r_n +\frac 18 \chi_1(\eps_n x) V(\eps_n) \tilde{\xi}(x, |u_n|)r_n+ f_{\epsilon_n}(x,|u_n|)u_n \nonumber  \\
& \qquad  \,\,\, +I^1_n(x)+I^\infty_n(x)) -\textnormal{i}\sum^3_{k=1} (\partial_k\eta_{n})\alpha_k
r_n. \nonumber
\end{align*}
At this point, proceeding as the proof of \cite[Lemma 4.18]{CG}, we are able to obtain the desired result.
For the convenience of the readers, we shall show the remaining proof. From \eqref{nmcnvbfhfyf66fyf} and \eqref{cvngfhfyf6fyfyfyyf}, it is easy to see that
\begin{align} \label{cnvnbhhgut77gytd}
|(-\textnormal{i} \alpha \cdot \nabla) \left(\eta_{n} r_n)| \leq
C\eta_{n} (|r_n|+ f_{\epsilon_n}(x,|u_n|)|u_n|+ 1 \right)
+|\nabla\eta_{n}| |r_n|,
\end{align}
where $C>0$ is a constant independent of $n$.
By the definition of $f_{\eps_n}$, we know that
\begin{align*}
|f_{\epsilon_n}(x, |t|)t|\leq C( |t|+|t|^2) \quad \mbox{for any}
\,\, x \in \R^3, t  \in \R.
\end{align*}
In virtue of \eqref{pd}, it then holds that
\begin{align*}
\begin{split}
|f_{\epsilon_n}(x,|u_n|)u_n| &\leq C\Big(|r_{n}|+|r_n|^2+\sum_{i\in\Lambda_1}( |U_{j,i}(\cdot-x_{j,i,n})|+ |U_{j,i}(\cdot-x_{j,i,n})|^2)\\
&\qquad+\sum_{i\in\Lambda_\infty}( \sigma_{j,i,n}
|U_{j,i}(\sigma_{j,i,n}(x-x_{j,i,n}))| +
\sigma_{j,i,n}^2|U_{j,i}(\sigma_{j,i,n}(x-x_{j,i,n}))|^2)\Big).
\end{split}
\end{align*}
According to Lemmas \ref{3cncbvggftd6ettd66d}, \ref{7cncbvggftd6ettd66d} and
\eqref{xcddnvbvh77fyfgttf4e}, we then deduce that, for any $x \in
\mathcal{A}^2_n$,
\begin{align*}
\eta_{n}|f_{\epsilon_n}(x,|u_n|)u_n| \leq C\eta_{n} \left
(|r_{n}|+|r_n|^2+1\right).
\end{align*}
This together with \eqref{cnvnbhhgut77gytd} yields that, for any $x \in \mathcal{A}^2_n$,
\begin{align*}
|(-\textnormal{i} \alpha \cdot \nabla)(\eta_{n} r_n)|\leq
C\eta_{n} \left(|r_n|+ |r_n|^2+1 \right)+|\nabla\eta_{n}||r_n|.
\end{align*}
As a consequence,
\begin{align*}
\int_{\R^3}|(-\textnormal{i} \alpha \cdot \nabla)(\eta_{n}r_n)|^p
\, dx \leq C \int_{\R^3}\eta_{n}^p \left(|r_n|^p+ |r_n|^{2p}+1
\right) \, dx + C\int_{\R^3} |\nabla\eta_{n}|^p|r_n|^p \, dx.
\end{align*}
By Lemma \ref{azscmbnbj88giuguugudd}, then, for any $p'>p \geq 2$,
there exist constants $C_p>0$ and $C_{p'}>0$ independent of $n$ such that
\begin{align} \label{xbcvfgtfyf66fyfy}
\begin{split}
\left(\int_{\R^3}|(-\textnormal{i} \alpha \cdot \nabla)(\eta_{n}r_n)|^p \, dx \right)^{\frac 1p }&\leq C_p+C_p\left(\int_{\R^3}|\nabla\eta_{n}|^p |r_n|^p \, dx \right)^{\frac 1p}\\
&\leq C_p+C_p\sigma^{1/2}_n\left (\int_{\mathcal{A}^2_n} |r_n|^p \, dx\right)^{\frac 1p}\\
&\leq C_p+C_p\sigma^{\frac{1}{2}(1-\frac{3}{p}+\frac{3}{p'})}_n \left(\int_{\mathcal{A}^2_n} |r_n|^{p'} \, dx \right)^{\frac {1}{p'}}\\
&\leq
C_p+C_pC_{p'}\sigma^{\frac{1}{2}(1-\frac{3}{p}+\frac{3}{p'})}_n.
\end{split}
\end{align}
For any $0<\nu<1/2$, choosing $p=3/(1-\nu)$ and $p'=3/\nu$ in
\eqref{xbcvfgtfyf66fyfy}, we then get that
\begin{align}\label{vcnbnbhghug77gugh}
\left(\int_{\R^3}|(-\textnormal{i} \alpha \cdot
\nabla)(\eta_{n}r_n)|^p \, dx\right)^{\frac 1 p}\leq
C_\nu\sigma^{\nu}_n.
\end{align}
It follows from Lemma \ref{di} that
\begin{align}\label{22cmvnbvbfhfyf66fyf}
\int_{\mathbb{R}^3}|\nabla (\eta_{n}r_n)|^p \, dx\leq C
\int_{\mathbb{R}^3}|(-\textnormal{i} \alpha \cdot \nabla)
(\eta_{n}r_n)|^p \,dx.
\end{align}
By Lemma \ref{azscmbnbj88giuguugudd}, \eqref{22cmvnbvbfhfyf66fyf} and \eqref{vcnbnbhghug77gugh}, we
then have that
\begin{align*}
\|\eta_n r_n\|_{W^{1,p}(\R^3)}\leq C''_\nu \sigma^\nu_n.
\end{align*}
Note that $p=3/(1-\nu)>3$, then
\begin{align}\label{xbcvgfgtd5tegdggdtd}
|r_{n}(x)|\leq C_\nu\sigma^{\nu}_n,\ x\in\mathcal{A}^3_n.
\end{align}
On the other hand,  by Lemma \ref{7cncbvggftd6ettd66d}, we know that
\begin{align}\label{bcvfgtfgrtyyfff}
|\sum_{i \in \Lambda_1} U_{j ,i}(x -x_{j, i, n})|\leq C, \
x\in\mathcal{A}^3_n,
\end{align}
where $C>0$ is independent of $n$. Moreover, by Lemma \ref{3cncbvggftd6ettd66d} and \eqref{xcddnvbvh77fyfgttf4e},
we have that
\begin{align}\label{ncbfhhgyfhyydyyddd}
|\sum_{i \in \Lambda_{\infty}} \sigma_{j, i, n} U_{j,
i}(\sigma_{j, i, n}(x-x_{j, i, n}))|\leq C, \ x\in\mathcal{A}^3_n,
\end{align}
where $C>0$ is independent of $n$. From \eqref{pd} and \eqref{xbcvgfgtd5tegdggdtd}-\eqref{ncbfhhgyfhyydyyddd},
we then get the result of this lemma, and the proof is
completed.
\end{proof}

\begin{lem}\label{2wncbvuuf8f77f7987}
Assume $5/2<q<3$, then, for any $0<\nu<1/2$, there exists a constant
$C_\nu>0$ independent of $n$ such that
$$
\int_{\mathcal{A}^4_n}|\nabla u_n|^2 \, dx \leq
C_\nu\sigma^{-1/2+\nu}_n,
$$
where \begin{align} \label{2221cn888834f6tfyfff}
\mathcal{A}^4_{n}=B_{(\overline{C}+7/2)\sigma^{-\frac{1}{2}}_n}(x_n)\setminus
B_{(\overline{C}+3/2)\sigma^{-\frac{1}{2}}_n}(x_n).
\end{align}
\end{lem}
\begin{proof}
Let $\eta_{n} \in C_0^{\infty}(\R^3, [0, 1])$ be a cut-off
function satisfying $\eta_n(x)=1$ for any $ x \in
{\mathcal{A}^4_n}$, $ \eta_n(x)=0$ for any $x \not \in
\mathcal{A}^3_n$ and $|\nabla\eta_{n}(x)|\leq C\sigma^{1/2}_n$ for
any $x \in \R^3$, where $C>0$ is independent of $n$. Since $u_n$ solves the equation
\begin{align*}
-\textnormal{i} \alpha \cdot \nabla u_n + a  \beta
u_n + V(\epsilon_nx)
u_n - \frac 18 \chi_1(\eps_n x) V(\eps_n x) \tilde{\xi}(x, |u_n|)u_n=f_{\epsilon_n}(x, |u_n|)
u_n,
\end{align*}
then
\begin{align}\label{22cbvnfhhfyf6ryyfyf}
\begin{split}
(-\textnormal{i} \alpha \cdot \nabla)(\eta_{n} u_n) &=\eta_{n}(-\textnormal{i} \alpha \cdot \nabla u_n)-\textnormal{i}\sum^3_{k=1} (\partial_k\eta_{n})\alpha_k u_n \\ 
&=\eta_{n}( - a \beta u_n - V(\epsilon_n x) u _{j,\epsilon_n}  + \frac 18 \chi_1(\eps_n x) V(\eps_n x) \tilde{\xi}(x, |u_n|)u_n \\ 
&\,\,\, \qquad+f_{\epsilon_n}(x, |u_n|)u_n)-\textnormal{i}\sum^3_{k=1} (\partial_k\eta_{n})\alpha_k
u_n.
\end{split}
\end{align}
Observe that $0 \leq \tilde{\xi}(x ,t) \leq 2$ and $|f_{\epsilon_n}(x, t)|\leq C(1+|t|)$ for
any $x \in \R^3$ and $t \geq 0$, it then follows from
\eqref{22cbvnfhhfyf6ryyfyf} that
\begin{align}\label{fhfhgyyghfyfttfrdf}
\begin{split}
 \hspace{-1cm} \int_{\R^3}|(-\textnormal{i} \alpha \cdot \nabla)(\eta_{n}
u_n)|^2 \, dx  \leq C \int_{\R^3}\eta_{n}^2 \left(|u_n|^2+
|u_n|^{4}\right) \, dx + \int_{\R^3}
|\nabla\eta_{n}|^2 |u_n|^2 \, dx.
\end{split}
\end{align}
Since, for any $0<\nu'<1/2,$ there exists a constant $C_{\nu'}>0$ such that
$|u_n(x)|\leq C_{\nu'}\sigma^{\nu'}_n$ for any $x\in
\mathcal{A}^3_n,$ see Lemma \ref{ncbvuuf8f77f7987},
$\mbox{supp} \, \eta_n\subset\mathcal{A}^3_n$ and $|\nabla\eta_{n}(x)|\leq C\sigma^{1/2}_n$ for
any $x \in \R^3$, then
\begin{align}\label{bfgvtgdttd77e5rterfdd}
 \int_{\R^3}\eta_{n}^2 \left(|u_n|^2+
|u_n|^{4}\right) \, dx + \int_{\R^3}
|\nabla\eta_{n}|^2 |u_n|^2 \, dx\leq
C_{\nu'}\sigma_n^{-1/2+2\nu'}.
\end{align}
From Lemma \ref{di}, we know that
\begin{align}\label{nchgyfy6fyft66s4srff}
\begin{split}
\int_{\mathcal{A}_n^4} |\nabla u_n|^2 \, dx &\leq C\Big(\int_{\R^3}  |\nabla(\eta_n u_n)|^2  \, dx  + \int_{\R^3}  |\nabla\eta_n|^2 |u_n|^2  \, dx \Big) \\
&\leq C \Big(\int_{\R^3} |(-\textnormal{i} \alpha \cdot
\nabla)(\eta_{n}u_n)|^2 \, dx + \int_{\R^3}
|\nabla\eta_n|^2 |u_n|^2  \, dx  \Big).
\end{split}
\end{align}
Combining  \eqref{fhfhgyyghfyfttfrdf}-\eqref{nchgyfy6fyft66s4srff}
and taking $\nu=2\nu'$, we then get the result of this lemma,
and the proof is completed.
\end{proof}

We are now ready to present a local Pohozaev identity with respect to $u_n$.
To do this, let us write $u_n=(u^1_n, u^2_n, u^3_n, u^4_n)$, $\beta=(a^0_{lm})_{4\times 4}$ and $\alpha_i=(a^i_{lm})_{4\times4}$ for $ i=1,2,3$.

\begin{lem}\label{ncbvhhyf6yyrt6f6ryr}
Let $B_n=B_{(\overline{C}+3)\sigma^{-{1}/{2}}_n}(x_n)$  and
$\psi\in C^\infty_0(\mathbb{R}, [0, 1])$ be such that $\psi(t)=1$
for any $t\leq (\overline{C}+2)\sigma^{-{1}/{2}}_n$, $\psi(t)=0$
for any $t\geq (\overline{C}+3)\sigma^{-{1}/{2}}_n$, $\psi'(t)\leq
0$ and $|\psi'(t)|\leq 2\sigma^{{1}/{2}}_n$ for any $t \in \R$.
Let $\varphi(x)=\psi(|x-x_n|)$ for any $x \in \R^3$. Then the
following identity holds,
\begin{align*}  \nonumber
&\int_{\R^3} \left(3F_{\epsilon_n}(x, |u_n|)
-f_{\epsilon_n}(x, |u_n|)|u_n|^2\right) \varphi  \,
dx+\epsilon_n\int_{\R^3}\left((x-x_n)
\cdot(\nabla_{x}F_{\epsilon_n})(x, |u_n|)\right) \varphi
\, dx\\ \nonumber
&\quad -\frac{a}{2}\int_{\R^3}\beta u_n \cdot \overline
{u_n}\varphi \, dx-\frac{1}{2}\int_{\R^3}V(\epsilon_n x)|u_n|^2
\varphi \, dx -\frac{\epsilon_n}{2}\int_{\R^3}((x-x_n)\cdot(\nabla
V)(\epsilon_n x))|u_n|^2\varphi  \,dx \\ \nonumber
&=\frac{a}{2}\int_{\R^3}((x-x_n)\cdot\nabla\varphi)(\beta u_n\cdot \overline{u_n}) \, dx+\frac{1}{2}\int_{\R^3}V(\epsilon_n x) \left((x-x_n)\cdot\nabla\varphi\right)  |u_n|^2 \, dx \\  \nonumber 
&\quad-\int_{\R^3}((x-x_n)\cdot\nabla\varphi) F_{\epsilon_n}(x,
|u_n|) \, dx  \\ \nonumber & \quad
-\frac{\textnormal{i}}{2}\sum^3_{i,j=1}\sum^4_{l,m=1}
a^i_{lm}\int_{\R^3}(x_j-x^j_n) u^l_n\partial_j\overline{u^m_n}
\partial_i\varphi \, dx
\end{align*}
\vspace{-0.4cm}
\begin{align}  \label{cnvhgyg7gyygqqagt}
\hspace{-6.5cm} + \, \frac{\textnormal{i}}{2}\sum^3_{i,j=1}\sum^4_{l,m=1}a^i_{lm}
\int_{\R^3}(x_j-x^j_n)u^l_n\partial_i\overline{u^m_n}
\partial_j\varphi \, dx.
\end{align}
\end{lem}
\begin{proof}
 Note that $u_n$ solves the equation
\begin{align*}
-\textnormal{i} \alpha \cdot \nabla u_n + a  \beta
u_n + V(\epsilon_nx)
u_n - \frac 18 \chi_1(\eps_n x) V(\eps_n x) \tilde{\xi}(x, |u_n|)u_n=f_{\epsilon_n}(x, |u_n|)
u_n.
\end{align*}
From Lemma \ref{xncbdggdtdtttd}, we know that, up to a
subsequence, $\lim_{n\rightarrow\infty}\epsilon_nx_{n}\in
\mathcal{M}^{\delta_0} \subset B_{R_0/2}(0)$. By the definition of
$\chi_1$, we then have that, for any $n \in \N^+$ large enough,
$\chi_1(\eps_n x)=0$ for any $x \in B_n$. Since supp $\varphi
\subset B_n$, then
\begin{align*}
&\mbox{Re}\int_{\R^3} \big(-\textnormal{i} \alpha \cdot \nabla u_n+ a \beta u_{n} + V(\epsilon_nx)
u_{n}, \, (x-x_n) \cdot \nabla u_n \varphi \big)_2 \, dx  \\
&= \mbox{Re} \int_{\R^3} \big(f_{\eps_n}(x, |u_n|) u_n, \, (x-x_n)
\cdot \nabla u_n \varphi \big)_2\, dx
\end{align*}
and
\begin{align*}
&\mbox{Re}\int_{\R^3}\big(-\textnormal{i} \alpha \cdot \nabla u_n+
a \beta u_{n} + V(\epsilon_nx) u_{n}, \,  u_n \varphi\big)_2 \, dx = \mbox{Re} \int_{\R^3} \big(f_{\eps_n}(x, |u_n|) u_n, \,  u_n \varphi \big)_2 \, dx,
\end{align*}
where $( \cdot, \, \cdot)_2$ denotes the usual inner product in $L^2(\R^3, \C^4)$. At this point, arguing as the proof of \cite[Lemma 4.20]{CG} and performing some standard calculations, we then get the result of the lemma, and the proof is completed.
\end{proof}

The proof of the lemma below can be completed by adapting the ideas to the proof of \cite[Lemma 4.21]{CG}.
For convenience of the readers, we shall present its proof here.

\begin{lem}\label{nvcbvjug8f877fs}
If $5/2 <q< 3$, then $\Lambda_\infty=\emptyset$.
\end{lem}
\begin{proof}
We argue by contradiction that $\Lambda_\infty\neq\emptyset$. Thus we are able to choose $x_n$ and $\sigma_n$
satisfy \eqref{mmmzcczdr4s4s4sss} and \eqref{mcnvbbbvuvjyyfffff}, respectively.
Note that Lemma \ref{ncbvhhyf6yyrt6f6ryr}, it hold that $|\nabla\varphi(x)|\leq 2\sigma^{1/2}_n$ for any $x \in
B_n$ and $\mbox{supp} \, \varphi \subset B_n$. Using Lemmas
\ref{ncbvuuf8f77f7987} and \ref{2wncbvuuf8f77f7987}, we then
deduce that, for any $0<\nu<1/4$, there exists a constant
$C_\nu>0$ such that the right hand side of
\eqref{cnvhgyg7gyygqqagt} satisfies that
\begin{align*}
&\bigg|-\frac{\textnormal{i}}{2}\sum^3_{i,j=1}\sum^4_{l,m=1}a^i_{lm}\int_{\R^3}(x_j-x^j_n)
u^l_n\partial_j\overline{u^m_n} \partial_i\varphi \, dx
+\frac{\textnormal{i}}{2}\sum^3_{i,j=1}\sum^4_{l,m=1}a^i_{lm}
\int_{\R^3}(x_j-x^j_n) u^l_n\partial_i\overline{u^m_n} \partial_j\varphi \, dx \nonumber\\
&\quad + \frac{a}{2}\int_{\R^3}((x-x_n)\cdot\nabla\varphi)(\beta u_n\cdot \overline{u_n}) \, dx+\frac{1}{2}\int_{\R^3}V(\epsilon_n x) \left((x-x_n)\cdot\nabla\varphi\right)  |u_n|^2 \, dx \nonumber\\
&\quad \left. -\int_{\R^3}((x-x_n)\cdot\nabla\varphi)
F_{\epsilon_n}(x, |u_n|) \, dx   \right|  \leq
C_\nu\sigma^{-1+2\nu}_n.
\end{align*}
We now turn to estimate the left hand side of \eqref{cnvhgyg7gyygqqagt}.
Observe first that, for any $n \in \N^+$ large enough,
\begin{align} \label{v33}
\begin{split}
\left|\frac{\epsilon_n}{2}\int_{\R^3}((x-x_n)\cdot(\nabla V)(\epsilon_n x))|u_n|^2\varphi  \,dx \right| &=\left|\frac{\epsilon_n}{2}\int_{B_n}((x-x_n)\cdot(\nabla V)(\epsilon_n x))|u_n|^2\varphi  \,dx \right| \\
& \leq \frac{\epsilon_n}{2} \left(\overline{C} +3\right) \sigma_n^{-1/2} \|\nabla V\|_{L^{\infty}(\overline{\mathcal{M}^{\delta_0+1}})} \int_{\R^3} |u_n|^2\varphi  \,dx,
\end{split}
\end{align}
where we used Lemmas \ref{xncbdggdtdtttd} and \ref{2xncbdggdtdtttd} for the inequality. By Lemma \ref{sbdd}, H\"older's inequality and \eqref{v33}, we then have that,for any $n \in \N^+$ large enough,
\begin{align*}
&\left|\frac{a}{2}\int_{\R^3}\beta u_n\cdot \overline{u_n} \varphi
\, dx +\frac{1}{2}\int_{\R^3}V(\epsilon_n x)|u_n|^2 \varphi \, dx
+\frac{\epsilon_n}{2}\int_{\R^3}((x-x_n)\cdot(\nabla
V)(\epsilon_n x))|u_n|^2\varphi  \,dx \right|\\
& \leq \left|\frac{a}{2}\int_{\R^3}\beta u_n\cdot \overline{u_n} \varphi
\, dx +\frac{1}{2}\int_{\R^3}V(\epsilon_n x)|u_n|^2 \varphi \, dx \right|+\left|
\frac{\epsilon_n}{2}\int_{\R^3}((x-x_n)\cdot(\nabla
V)(\epsilon_n x))|u_n|^2\varphi  \,dx \right| \\
 &\leq C\int_{\R^3}|u_n|^2 \varphi \, dx
\leq  C\sigma^{-1/2}_n \|u_n\|_{L^3(\R^3)}^2 \leq  C\sigma^{-1/2}_n.
\end{align*}
Moreover, from Lemma \ref{sbdd} and the definition of $F_{\eps_n}$, we obtain that
\begin{align*}
\left|\epsilon_n\int_{\R^3}\left((x-x_n)
\cdot(\nabla_{x}F_{\epsilon_n})(x, |u_n|)\right) \varphi
\, dx\right|\leq C\sigma^{-1/2}_n\int_{\R^3}(|u_n|^q+|u_n|^3) \varphi\,
dx\leq C\sigma^{-1/2}_n.
\end{align*}
In virtue of the definitions of $f_{\eps_n}$ and $F_{\eps_n}$, there holds that
\begin{align}\label{mmvcnvii8g9fff}
\begin{split}
&\int_{\R^3}\left(3F_{\epsilon_n}(x, |u_n|) - f_{\epsilon_n}(x, |u_n|)|u_n|^2 \right) \varphi \, dx \\ 
&=\int_{\R^3}\left(1- \chi_2(\epsilon_n x)\right) \Big( \frac
{3}{q}|u_n|^q + |u_n|^{q} \left(m_{\epsilon_n}(|u_n|^2)
\right)^{\frac{3-q}{2}} \Big) \varphi \, dx \\ 
&\quad-\int_{\R^3}\left(1- \chi_2(\epsilon_n
x)\right)\Big(|u_n|^{q} + \frac{q}{3} |u_n|^{q}
\left(m_{\epsilon_n}(|u_n|^2)
\right)^{\frac{3-q}{2}}\Big) \varphi \, dx \\ 
&\quad-\frac{3-q}{3} \int_{\R^3}\left(1- \chi_2(\epsilon_n x)\right) |u_n|^{q+2}
\left(m_{\epsilon_n}(|u_n|^2) \right)^{\frac{3-q}{2} -1} b_{\epsilon_n}(|u_n|^2) \varphi \, dx \\ 
&\quad+\int_{\R^3} \chi_2(\epsilon_n x) \left(3 G_{\eps_n}(x,
|u_n|)-g_{\eps_n}(x, |u_n|)|u_n|^2\right) \varphi \, dx.
\end{split}
\end{align}
Since $tb_{\epsilon_n}(t)\leq m_{\epsilon_n}(t)$ for any $t \geq 0$, see
Lemma \ref{mprop}, and
\begin{align}
\left |\int_{\R^3}\chi_2(\epsilon_n x) \left(3 G_{\eps_n}(x,
|u_n|)-g_{\eps_n}(x, |u_n|)|u_n|^2\right) \varphi \, dx \right|
\leq C\int_{\R^3}|u_n|^2 \varphi \, dx \leq
C\sigma^{-1/2}_n,\nonumber
\end{align}
where we used the definitions of $g_{\eps_n}$ and $G_{\eps_n}$, it then yields from \eqref{mmvcnvii8g9fff} that
\begin{align*}
&\int_{\R^3}\left(3F_{\epsilon_n}(x, |u_n|) - f_{\epsilon_n}(x, |u_n|)|u_n|^2 \right) \varphi \, dx 
\geq\frac {3-q}{q}\int_{\R^3}\left(1- \chi_2(\epsilon_n x)\right)
|u_n|^q \varphi\, dx -C\sigma^{-1/2}_n.
\end{align*}
In view of Lemma \ref{xncbdggdtdtttd}, we have that, up to a
subsequence,
$x_{j,i_\infty}:=\lim_{n\rightarrow\infty}\epsilon_nx_{n}\in
\mathcal{M}^{\delta_0}$. It then follows that
$1-\chi_2(x_{j,i_\infty})>0$ and for any $n \in \mathbb{N}^+$
large enough,
\begin{align*}
\int_{\R^3}\left(3F_{\epsilon_n}(x, |u_n|) - f_{\epsilon_n}(x,
|u_n|)|u_n|^2 \right) \varphi \, dx \geq\frac
{3-q}{2q}(1-\chi_2(x_{j,i_\infty}))\int_{\R^3}|u_n|^q \varphi \,
dx-C\sigma^{-1/2}_n.
\end{align*}
Taking into account \eqref{cnvhgyg7gyygqqagt}, we now derive from
the above arguments that
\begin{align}\label{cnvbvhhfyf6vyyvyv}
\int_{\R^3}|u_n|^q \varphi \, dx \leq C\sigma^{-1/2}_n.
\end{align}
Let us now denote $G_n:=B_{L\sigma^{-1}_n}(x_n)$ and take $L>0$ is large enough
such that
\begin{align*}
\int_{B_L(0)}|U_{j, i_\infty}|^{q} \, dx \geq \frac{C_*}{2}
\end{align*}
where $C_*:=\int_{\R^3}|U_{j,i_\infty}|^{q}\,dx>0$. Thus, for any $n
\in \mathbb{N}^+$ large enough,
$$
G_n\subset B_{(\overline{C}+2)\sigma^{-\frac{1}{2}}_n}(x_n)\subset
B_n
$$
and $\varphi(x)=1$ for any $x \in G_n$. Since $ \sigma^{-1}_n
u_n(\sigma^{-1}_n\cdot+x_{n})\rightharpoonup U_{j,i_\infty}$ in
$\dot{H}^{1/2}(\R^3,\C^4)$ as $n \to \infty$, see Lemma
\ref{cmvnvhhguf7fuufss}, we then obtain that, for any $n \in
\mathbb{N}^+$ large enough,
\begin{align}\label{ncbvuuf78f77fesd}
\begin{split}
\int_{B_n}|u_{n}|^q\varphi \, dx &\geq \int_{G_n}|u_{n}|^q \,
dx=\sigma^{q-3}_n\int_{B_L(0)}|\sigma^{-1}_{n}u_{n}(\sigma^{-1}_nx+x_n)|^q\,
dx \\ &=\sigma^{q-3}_n\int_{B_L(0)}|U_{j,i_\infty}|^q \, dx +
o_n(1)\sigma^{q-3}_n\geq\frac{C_*}{4}\sigma^{q-3}_n.
\end{split}
\end{align}
Combining  \eqref{cnvbvhhfyf6vyyvyv} and \eqref{ncbvuuf78f77fesd}
leads to
\begin{align}\label{ncbvnvhgyfhfhf}
\sigma^{q-3}_n\leq C\sigma^{-1/2}_n.
\end{align}
It is impossible, because of $ q>5/2$. This in turn indicates that
$\Lambda_{\infty}= \emptyset$, and the proof is completed.
\end{proof}

We are now in a position to prove Proposition \ref{bcbvhfyfyufuadx}.
\medskip

\noindent{\bf Proof of Proposition \ref{bcbvhfyfyufuadx}.} Let
$\epsilon_n\rightarrow 0^+$ as $n\rightarrow\infty$. From \eqref{pd} and Lemma
\ref{nvcbvjug8f877fs}, we then obtain that
$$
u_{j, \eps_n}=\sum_{i \in \Lambda_1} U_{j ,i}(\cdot -x_{j, i,n})+ r_n.
$$
In view of $(4)$ of Lemma \ref{cmvnvhhguf7fuufss} and Lemma \ref{7cncbvggftd6ettd66d}, we further know that, for any $\varsigma>0$, there exists a constant $\delta>0$ independent of $n$ such that, for any $n\in\N^+$ large enough,
\begin{align}\label{jcmvnhhvufyfyy6dr}
\sup_{y\in\mathbb{R}^3}\int_{B_{\delta}(y)}|u_{j, \eps_n}|^{3} \, dx<\varsigma.
\end{align}
For simplicity, we shall write $u_n=u_{j, \eps_n}$ in the following. For any $y\in\R^3$, we choose a cut-off function $\eta \in C_0^{\infty}(\R^3, [0, 1])$ such that $\eta (x)=1$ for any $x \in
B_{\delta/2}(y)$, $\eta (x)=0$ for any $x \not\in B_{\delta}(y) $
and $|\nabla\eta(x)|\leq 4/\delta$ for any $x \in \R^3$. Note that
\begin{align}\label{2wcbvnfhhfyf6ryyfyf}
\begin{split}
-\textnormal{i} \alpha \cdot \nabla(\eta u_n)
&=\eta(- \textnormal{i} \alpha \cdot \nabla u_n)-\textnormal{i}\sum^3_{k=1} (\partial_k\eta) \alpha_k u_n \\ 
&=\eta\big(- a  \beta u_n - V(\epsilon_nx) u_n+ \frac 1 8 \chi_1(\eps_nx) V(\eps_n x) \tilde{\xi}(x, |u_n|) u_n \\ 
&\,\,\, \qquad + f_{\epsilon_n}(x, |u_n|) u_n\big)-\textnormal{i}\sum^3_{k=1} (\partial_k\eta) \alpha_k u_n.
\end{split}
\end{align}
By \eqref{2wcbvnfhhfyf6ryyfyf}, Lemma \ref{di} and the fact that $0 \leq \tilde{\xi}(x, t) \leq 2$ and $|f_{\epsilon_n}(x, t)|\leq
C(1+|t|)$ for any $x \in \R^3$ and $t \geq 0$, we then obtain, for
any $p>1$,
\begin{align}\label{cncbfggfyr6fttf}
\|\eta u_n\|_{W^{1,p}(\R^3)}\leq
C_p(\|u_n\|_{L^p(B_\delta(y))}+\|\eta
|u_{n}|^2\|_{L^p(B_\delta(y))}).
\end{align}
For any $1<p<3$, using H\"older's inequality and
\eqref{jcmvnhhvufyfyy6dr}, we have that
\begin{align}\label{hvnvhgyf88fiffl}
\| \eta
|u_n|^2\|_{L^p(B_\delta(y))}&\leq\|u_n\|_{L^3(B_\delta(y))}\|\eta
u_n\|_{L^{p^*}(B_\delta(y))} \leq\varsigma^{1/3}\|\eta
u_n\|_{L^{p^*}(B_\delta(y))},
\end{align}
where $p^*=3p/(3-p)$. Taking into account \eqref{cncbfggfyr6fttf},
\eqref{hvnvhgyf88fiffl}, the Sobolev inequality and Lemma
\ref{sbdd}, we then deduce that, for any $\varsigma>0$ small
enough,
\begin{align*}
\|u_n\|_{L^{p^*}(B_{\delta/2}(y))} \leq \|\eta
u_n\|_{L^{p^*}(B_{\delta}(y))}\leq C_p\|
u_n\|_{L^p(B_\delta(y))}\leq C_p\|u_n\|_{L^3(B_\delta(y))} \leq
C_p.
\end{align*}
Note that $ p^* \rightarrow+\infty$ as $p\rightarrow 3^-$, then it
is not difficult to deduce that
\begin{align*}
\|u_n\|_{W^{1,4}(B_{\delta/4}(y)))} &\leq
C(\|u_n\|_{L^4(B_{\delta/2}(y))}+\|
u_n\|_{L^8(B_{\delta/2}(y))}^2) \leq C.
\end{align*}
Therefore, there exists  a constant $M_N>0$ such that
$\sup_{x\in\mathbb{R}^3}|u_n(x)|<M_N.$ Thus we have completed the
proof. \hfill$\Box$

\section{Proof of theorem \ref{jgh77rtff11}}\label{proof}

\noindent{\bf Proof of Theorem \ref{jgh77rtff11}.} By Proposition
\ref{bcbvhfyfyufuadx}, we know that there exists a constant  $\epsilon'''_k>0$ such
that, for any $0<\epsilon<\epsilon'''_k$,
\begin{align}\label{bcvvuuvyctttc}
m_\epsilon(|u_{j,\epsilon}|^2)=|u_{j,\epsilon}|^2, \quad 
b_\epsilon(|u_{j,\epsilon}|^2)=1.
\end{align}
Due to $\Lambda_\infty=\emptyset$, see Lemma
\ref{nvcbvjug8f877fs}, and \eqref{bcvvuuvyctttc}, with the help of
the profile decomposition \eqref{pd}, then the situation is the
same as the case of subcritical equations treated in \cite{WZ1}.
Arguing as the proof of \cite[Lemma 4.5]{WZ1}, we can deduce that,
for any $\delta>0$, there exist constants $c=c(\delta, N)$ and
$C=C(\delta, N)>0$ such that
\begin{align*}
|u_{j,\epsilon}(x)|\leq C\exp\left(-c\,\left(\mbox{dist}(x, (\mathcal{V}^\delta)_\epsilon))^{\frac{2-\tau}{2}}\right)\right) \quad
\mbox{for any} \,\, x\in \mathbb{R}^3.
\end{align*}
This then shows that $\chi_1(\eps x) \tilde{\xi}(x,
|u_{j,\eps}|)=0$ and $ f_\epsilon(x,
|u_{j,\epsilon}|)=|u_{j,\epsilon}|^{q-2}+|u_{j,\epsilon}|$, which
suggests that $u_{j,\epsilon}$ are actually solutions of
\eqref{Deps}. By making a change of variable, we then obtain
Theorem \ref{jgh77rtff11}, and the proof is completed.
\hfill$\Box$

\bigskip

{\noindent \bf Acknowledgements.} We should like to thank the anonymous referees for his/her careful readings of our manuscript
and the useful comments and suggestions to improve the manuscript.


\begin{thebibliography}{10}

\bibitem{ABC} {A. Ambrosetti, M. Badiale, S. Cingolani:} {\it Semiclassical states of nonlinear Schr\"odinger equations,} Arch. Rational Mech. Anal. 140 (3) (1997)  285-300.

\bibitem{AFM} {A. Ambrosetti, V. Felli, A. Malchiodi:} {\it Ground states of nonlinear Schr\"odinger equations with potentials vanishing at infinity,} J. Eur. Math. Soc. (JEMS) 7(1) (2005) 117-144.

\bibitem{AMN} {A. Ambrosetti, A. Malchiodi, W.-M. Ni:} {\it Singularly perturbed elliptic equations with symmetry: existence of solutions concentrating on spheres I,} Commun. Math. Phys. 235(3) (2003) 427-466.

\bibitem{AR} {V. Ambrosio, V.D.  R\v{a}dulescu:} {\it Fractional double-phase patterns: concentration and multiplicity of solutions}, J. Math. Pures Appl. (9) 142 (2020) 101-145.








\bibitem{BD} {T. Bartsch, T. Ding:} {\it Deformation theorems on non-metrizable vector spaces and applications to critical point theory,} Math. Nachr. 279 (12) (2006) 1267-1288.



\bibitem{BDr} {J.D. Bjorken, S.D. Drell:} {\it Relativistic Quantum Fields,} McGraw-Hill, 1965.



\bibitem{bo} {W. Borrelli, William; R. L. Frank, Rupert:} {\it Sharp decay estimates for critical Dirac equations,} Trans.
Amer. Math. Soc. 373 (3) (2020) 2045-2070.



\bibitem{BJ} {J. Byeon, L. Jeanjean:} {\it Standing waves for nonlinear Schr\"odinger equations with a general nonlinearity,} Arch. Ration. Mech. Anal. 185 (2) (2007) 185-200.

\bibitem{BW1} {J. Byeon, Z.-Q. Wang:} {\it Standing waves with a critical frequency for nonlinear Schr\"odinger equations,} Arch. Ration. Mech. Anal. 165 (4) (2002) 295-316.

\bibitem{BW2} {J. Byeon, Z.-Q. Wang:} {\it Standing waves with a critical frequency for nonlinear Schr\"odinger equations. II,}  Calc. Var. Partial Differential Equations 18 (2003) (2) 207-219.

\bibitem{CG} {S. Chen, T, Gou:} {\it Infinitely many localized semiclassical states for critical nonlinear Dirac equations,} Nonlinearity 34 (2021) 6358-6397.



\bibitem{CLW} {S. Chen, J. Liu, Z.-Q. Wang:} {\it Localized nodal solutions for a critical nonlinear Schr\"odinger equation,} J. Funct. Anal. 277 (2) (2019) 594-640.

\bibitem{CW} {S. Chen, Z.-Q. Wang:} {\it Localized nodal solutions of higher topological type for semiclassical nonlinear
Schr\"odinger equations,} Calc. Var. Partial Differential Equations (2017) 56:1.

\bibitem{CDX}{Y. Chen, Y. Ding, T. Xu:} {\it Potential well and multiplicity of solutions for nonlinear Dirac equations,}
Commun. Pure Appl. Anal.  19 (1) (2020) 587-607.

\bibitem{DF1} {M. del Pino, P.L. Felmer:} {\it Multi-peak bound states for nonlinear Schr\"odinger equations,} Ann. Inst. H. Poincar\'e Anal. Non Lin\'eaire 15 (2) (1998) 127-149.

\bibitem{DF2} {M. del Pino, P.L. Felmer:} {\it Local mountain passes for semilinear elliptic problems in unbounded domains,} Calc. Var. Partial Differential Equations 4 (2) (1996) 121-137.

\bibitem{Ding} Y. Ding: {\it Variational Methods for Strongly Indefinite Problems,} Interdisciplinary Mathematical Sciences-Vol. 7, World Scientific Publ. Singapore (2007).

\bibitem{Ding1} {Y. Ding:} {\it Semi-classical ground states concentrating on the nonlinear potential for a Dirac equation,}  J. Differential Equations 249 (5) (2010) 1015-1034.

\bibitem{DiLi} {Y. Ding, X. Liu:} {\it Semi-classical limits of ground states of a nonlinear Dirac equation,} J. Differential Equations 252 (9) (2012) 4962-4987.



\bibitem{DiRu1} {Y. Ding, B. Ruf:} {\it Existence and concentration of semiclassical solutions for Dirac equations with critical nonlinearities,} SIAM J. Math. Anal. 44 (6) (2012) 3755-3785.

\bibitem{DLR} {Y. Ding, C. Lee, B. Ruf:} {\it On semiclassical states of a nonlinear Dirac equation}, Proc. Roy. Soc. Edinburgh Sect. A 143 (4) (2013) 765-790.



\bibitem{DX} {Y. Ding, T. Xu:} {\it Localized concentration of semi-classical states for nonlinear Dirac equations,} Arch. Ration. Mech. Anal. 216 (2) (2015) 415-447.







\bibitem{FLR} {R. Finkelstein, R. LeLevier, M. Ruderman:} {\it Nonlinear spinor fields;} Phys. Rev. 83 (2) (1951) 326-332.

\bibitem{FFK} {R. Finkelstein, C. Fronsdal, P. Kaus:} {\it Nonlinear spinor field,} Phys. Rev. 103(5) (1956) 1571-1579.



\bibitem{IS} {T. Ichinose, Y. Saito:}  {\it Improved Sobolev embedding theorems for vector-valued functions,} Funkcial. Ekvac. 57 (2) (2014) 245-295.

\bibitem{JT} {L. Jeanjean, K. Tanaka:} {\it Singularly perturbed elliptic problems with superlinear or asymptotically linear nonlinearities,} Calc. Var. Partial Differential Equations  21 (3) (2004) 287-318.

\bibitem{JR} {C. Ji, V.D. R\v{a}dulescu:} {\it Multiplicity and concentration of solutions to the nonlinear magnetic Schr\"odinger equation}, Calc. Var. Partial Differential Equations 59 (4) (2020) Paper No. 115, 28 pp.



\bibitem{MV} {V. Moroz, J. Van Schaftingen:} {\it Semiclassical stationary states for nonlinear Schr\"odinger equations with fast decaying potentials,} Calc. Var. Partial Differential Equations 37 (1-2) (2010) 1-27.

\bibitem{Morse} {J. Moser:} {\it A new proof of De giorgi's theorem concerning the regularity problem
for elliptic differential equations,} Comm. Pure Appl. Math. 13 (3) (1960) 457-468.

\bibitem{Oh1} {Y.-G. Oh:} {\it Existence of semiclassical bound states of nonlinear Schr\"odinger equations with potentials of the class $(V )_a$, }Comm. Partial Differential Equations 13.12 (1988)  1499-1519.

\bibitem{Oh2} {Y.-G. Oh:} {\it On positive multi-lump bound states of nonlinear Schr\"odinger equations under multiple well potential,} Comm. Math. Phys. 131 (2) (1990) 223-253.

\bibitem{Ra} {H. Rabinowitz:} {\it On a class of nonlinear Schr\"odinger equations,} Z. Angew. Math. Phys. 43 (2) (1992) 270-291.

\bibitem{Wang} {X. Wang:} {\it On concentration of positive bound states of nonlinear Schr\"odinger equations,} Comm. Math. Phys. 153 (2) (1993) 229-244.

\bibitem{WZ} {Z.-Q. Wang, X. Zhang:} {\it An infinite sequence of localized semiclassical bound states for nonlinear Dirac equations,} Calc. Var. Partial Differential Equations (2018) 57:56.

\bibitem{WZ1} {Z.-Q. Wang, X. Zhang:} {\it Semiclassical states of nonlinear Dirac equations with degenerate potential,} Ann. Mat. Pura Appl. (4) 198 (6) (2019) 1955-1984.

\bibitem{Th} {B. Thaller:} {\it The Dirac Equation,} Texts and Monographs in Physics, Springer, Berlin, 1992.

\bibitem{SST}{C. Tintarev:} {\it Concentration analysis and cocompactness,} in Concentration Analysis and Applications to PDE, 117-141, Trends Math., Birkh\"auser/Springer, Basel, 2013.

\bibitem{ZLL} {J. Zhao, X. Liu, J. Liu:} {\it $p$-Laplacian equations in $\R^N$ with finite potential via truncation method, the critical case,} J. Math. Anal. Appl. 455 (2017) 58-88.

\end{thebibliography}
\end{document}